\documentclass{article}
\usepackage{t1enc,amsfonts,latexsym,amsmath,amsthm,verbatim,amssymb,graphics,mathdots,pstool,color,wrapfig,graphicx,xcolor,hyperref}
\usepackage{tikz}
\usepackage{pgfplots}
\usepackage[ruled,vlined]{algorithm2e}
\newcommand{\rank}{\mathsf{rank}}

\renewcommand{\Re}[1]{\mathsf{Re}(#1)}

\DeclareMathOperator*{\argmin}{arg\,min }

\newcommand{\scal}[1]{\left \langle #1 \right \rangle}
\newcommand{\para}[1]{\left( #1\right)}

\newtheorem{theorem}{Theorem}[section]
\newtheorem{proposition}[theorem]{Proposition}
\newtheorem{lemma}[theorem]{Lemma}

\def\re{{\mathsf {Re}}}
\def\A{{\mathcal A}}

\def\V {{\mathcal V }}
\def\M{{\mathbb M}}

\def\G{{\mathcal G}}

\def\R{{\mathbb R}}
\def\S{{\mathcal{S}}}
\def\Q{{\mathcal{Q}}}

\def\ep{{{ \epsilon}}}

\def\A{{\mathcal{A}}}

\newcommand{\skal}[1]{\langle #1 \rangle}

\title{An unbiased approach to low rank recovery}
\author{Marcus Carlsson\footnotemark[1] \and Daniele Gerosa\footnotemark[1] \and Carl Olsson \thanks{Centre for Mathematical Sciences , Lund University, Box 118, SE-22100, Lund,  Sweden,  {calle@maths.lth.se, marcus.carlsson@math.lu.se,daniele.gerosa@math.lu.se}}
\thanks{Electrical Engineering, Chalmers University of Technology, caols@chalmers.se}}
\date{}

\begin{document}

\maketitle
\begin{abstract}
Low rank recovery problems have been a subject of intense study in recent years.
While the rank function is useful for regularization it is difficult to optimize due to its non-convexity and discontinuity.
The standard remedy for this is to exchange the rank function for the convex nuclear norm, which is known to favor low rank solutions under certain conditions. On the downside the nuclear norm exhibits a shrinking bias that can severely distort the solution in the presence of noise, which motivates the use of stronger non-convex alternatives.
In this paper we study two such formulations. We characterize the critical points and give sufficient conditions for a low rank stationary point to be unique. Moreover, we derive conditions that ensure global optimality of the low rank stationary point and show that these hold under moderate noise levels.
\end{abstract}

\section{Introduction}\label{sec:intro}
Recovering a low rank matrix from noisy measurements is a problem that is frequently occurring in many applications.
Typically we are trying to recover a matrix $X$ from a set of noisy observations $\A X \approx b$ of linear combinations of the elements in $X$. Here $\A$ is a linear operator $\M_{n_1,n_2} \mapsto \R^m$, where $\M_{n_1,n_2}$ is the set of matrices of size $n_1\times n_2$ with real or complex coefficients, and $b \in \R^m$.
The linear system is often vastly under-determined and therefore regularization in the form of a rank penalty
or constraint is usually applied, for example one may consider minimization of
\begin{equation}
\mu \text{rank} (X) + \|\A X - b\|^2_2
\label{eq:mupenalty}
\end{equation}
where $\mu$ is a parameter whose size determines the rank of the global minimizer (or minimizers) of \eqref{eq:mupenalty}. If the desired rank, say $K$, is known a priori one ideally wishes to solve \begin{equation}\label{best}\argmin_{\text{rank} X\leq K}\|\A X-b\|_2,\end{equation} whose solution, whenever it is unique, we will name the ``best rank $K$ solution''. However, due to the non-linearity of the manifold  $R_K = \{X\in\R^{n_1 \times n_2}; \text{rank}(X)\leq K\}$, this problem can be very difficult, and there are no algorithms that are guaranteed to find the global minima (for all $\A$). The same holds for \eqref{eq:mupenalty} since $\rank(X)$ is both non-convex and discontinuous.

In this paper we provide algorithms for solving both \eqref{eq:mupenalty} and \eqref{best}, along with theory proving that they work, given that $\A$ satisfies certain restrictions. In order to put both theorems in the same framework, note that \eqref{best} can be rewritten as a global optimization problem as follows
\begin{equation}
\iota_{R_K}(X) + \|\A X - b\|^2_2,
\label{eq:fixedrank}
\end{equation}
where
\begin{equation}
\iota_{R_K}(X) =
\begin{cases}
0 & X \in R_K \\
\infty & \text{otherwise}
\end{cases}\label{iotaRk}
\end{equation}

The by now standard approach to retrieve approximate solutions to the above problems is to replace the rank function with the convex nuclear norm $\|X\|_* = \sum_i \sigma_i(X)$ (where $\sigma(Y)$ denotes the singular values, see e.g.~\cite{candes-etal-acm-2011,recht-etal-siam-2010}), resulting in the relaxation
\begin{equation}
\lambda \|X\|_*+\|\A X - b\|^2_2.
\label{eq:nuclearobj}
\end{equation}
where $\lambda$ is a parameter that can be used to tune the solution until the desired rank restriction is met.
In \cite{recht-etal-siam-2010} the notion of restricted isometry property (RIP) was introduced to the matrix setting; RIP holds for the operator $\A$ if it fulfills
\begin{equation}
(1-\delta_K)\|X\|_F^2 \leq \|\A X\|^2_2 \leq (1+\delta_K)\|X\|_F^2,
\label{eq:matrisRIP}
\end{equation}
for all $X$ with $\text{rank}(X) \leq K$.
Since then a number of generalizations that give performance guarantees for the nuclear norm relaxation have appeared \cite{oymak2011simplified,candes-etal-acm-2011,candes2009exact}.

While the convexity of the nuclear norm simplifies inference it also introduces a shrinking bias \cite{oymak-etal-2015,mohan2010iterative,larsson-olsson-ijcv-2016,grussler-etal-arxiv-2016,hu-etal-pami-2013,oh-etal-pami-2016,canyi2015,gu-2016}; the rank function assigns a constant penalty to any non-zero singular value, independently of its size, whereas the nuclear norm penalty is explicitly based on the magnitude of the singular values.
In high noise settings, where a large regularization weight $\lambda$ is required, \eqref{eq:nuclearobj} often produce solutions that are far from the ground truth \cite{oymak-etal-2015,mohan2010iterative,larsson-olsson-ijcv-2016,grussler-etal-arxiv-2016,hu-etal-pami-2013,oh-etal-pami-2016,canyi2015,gu-2016}.
Thus researchers have designed algorithms for non-convex formulations \cite{oymak-etal-2015,mohan2010iterative,hu-etal-pami-2013,oh-etal-pami-2016,canyi2015,gu-2016}. These methods however usually only guarantee convergence to a stationary or locally optimal point.
In \cite{larsson-olsson-ijcv-2016,parekh-selesnick-spl-2016,grussler-etal-arxiv-2016} it was observed that it is sometimes possible to use a non-convex regularizer and still get a convex problem, when the data term is sufficiently convex. For example, \cite{larsson-olsson-ijcv-2016} showed that the (lower semi-continuous) convex envelope of
\begin{equation}
\mu \text{rank}(X) + \|X-M\|_F^2,
\label{eq:murankapprox}
\end{equation}
is
\begin{equation}
\sum_i r_\mu(\sigma_i(X)) + \|X-M\|_F^2,
\label{eq:murankconvenv}
\end{equation}
where $r_\mu(\sigma) = \mu - \max(\sqrt{\mu}-\sigma,0)^2$. In particular, the global optimizers of \eqref{eq:murankapprox} and \eqref{eq:murankconvenv} are the same (assuming \eqref{eq:murankapprox} has a unique solution).
The more general problem
\begin{equation}
\sum_i r_\mu(\sigma_i(X)) + \|\A X-b\|^2_2,
\label{eq:muregprob}
\end{equation}
is not necessarily convex, but it follows from the work in \cite{carlsson2018convex} that it has the same global minimizers as \eqref{eq:mupenalty}
if $\|\A\| < 1$, where $\|\A\|$ denotes the operator norm of $\A$. This is because it turns out that the first term in \eqref{eq:muregprob} is the so called ``quadratic envelope'' of $\mu \rank(X)$ (see Theorem~\ref{celo1} in \cite{carlsson2016convexification} for details).

In this paper we will further study the problem \eqref{eq:muregprob} and its relation to \eqref{eq:mupenalty}, and simultaneously derive an analogous theory for \eqref{eq:fixedrank} and the corresponding expression with $\iota_{R_K}$ replaced by its quadratic envelope. We study the distribution of stationary points of these problems and show that under certain conditions the low rank stationary points are unique. We then give additional conditions that ensure that the low rank stationary point actually equals the best rank $K$ solution (for a suitable $K$), and finally show that these are fulfilled as long as the noise level is not severe. The theorems, which are briefly presented in Section \ref{sec key}, are based on concrete estimates as opposed to the by now usual asymptotic probabilistic arguments which give results that usually apply for very large matrix sizes. The results are analogous to those presented in \cite{carlsson2020unbiased}, where the vector counterpart of the above problems was studied.

\subsection{Shrinking Bias}\label{sec:bias}
Before we present our theoretical results we first give a brief explanation of the shrinking bias of the nuclear norm which motivates the use of non-convex regularizers. First consider the problem of minimizing \eqref{eq:murankapprox}
using the relaxations \eqref{eq:murankconvenv} and
\begin{equation}
2\sqrt{\mu}\|X\|_* + \|X-M\|_F^2.
\label{eq:nuclearprox}
\end{equation}
In both of these cases a closed form solution can be obtained from the SVD of $M$.
In the first case \eqref{eq:murankconvenv} the solution is the so called hard thresholding of $M$. More precisely, if $M = U \Lambda_{\sigma(M)} V^T$ is a singular value decomposition of $M$, where $\Lambda_{\sigma(M)}$ is a diagonal matrix containing the singular values $\sigma(M)$, then the solution is given by $X = U \Lambda_{\sigma(X)} V^T$, where $\sigma_i(X)$ equals $\sigma_i(M)$ for all indices $i$ such that $\sigma_i(M)>\sqrt{\mu}$, zero for all indices such that the reverse inequality holds, and any number in the interval $[0,\sqrt{\mu}]$ if it happens that $\sigma_i(M)=\sqrt{\mu}$ (see e.g.~\cite{larsson-olsson-ijcv-2016} for details).
Note that, whenever $\sigma_i(M)\neq \sqrt{\mu}$ for all $i$, this is also the solution of the original unrelaxed formulation \eqref{eq:murankapprox}.
For \eqref{eq:nuclearprox} we instead get the so called soft thresholding \cite{cai-etal-2008}, given by
\begin{equation}
\sigma_i(X) = \begin{cases}
0 & \sigma_i(M)  < \sqrt{\mu}\\
\sigma_i(M)-\sqrt{\mu} & \text{otherwise}\\
\end{cases}.
\end{equation}
Here we have chosen the regularization weights, $\mu$ and $2\sqrt{\mu}$ respectively, so that the two methods are able to suppress an equal amount of noise (which we assume is what accounts for the singular values that are smaller than $\sqrt{\mu}$).
However to suppress this level of noise the nuclear norm has to subtract equally much from the large singular values (that corresponds to the matrix we want to recover).

\begin{figure}
\includegraphics[width=60mm]{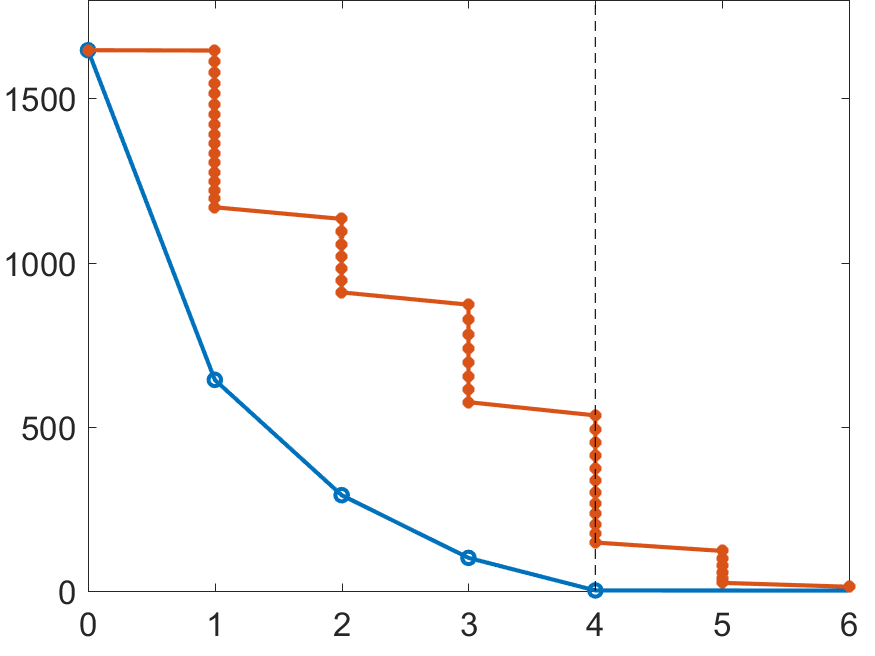}
\includegraphics[width=60mm]{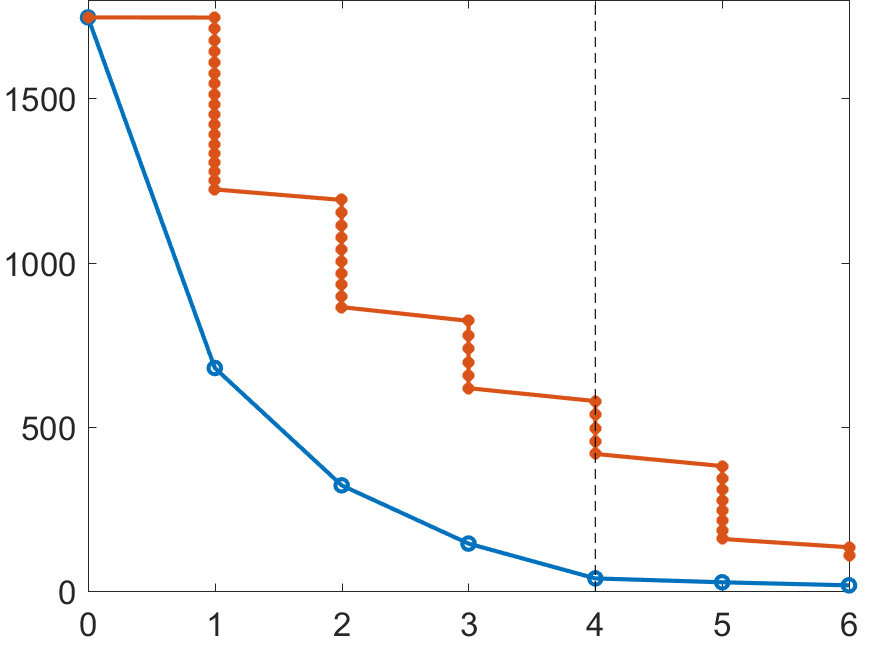}
\caption{Rank (x-axis) vs. data fit $\|\A X- b\|_F^2$ (y-axis) using \eqref{eq:nuclearobj} (orange curves) and \eqref{eq:muregprob} (blue curves). The noise is i.i.d. Gaussian with std $0.1$ (left) and $0.5$ (right). Each run with a different $\mu$ or $\lambda$ is represented by a dot. The reason this is not visible in the blue curve is that for different values of $\mu$, the algorithm finds the same point as long as the rank does not change. }
\label{fig:nncomp}
\end{figure}

For the above example with the data term $\|X-M\|_F^2$ the matrices $U$ and $V$ in the SVD of $X$ and $M$ are the same, and the problem essentially simplifies to a vector problem involving only the singular values. With a more general data term of the type $\|\A X -b\|^2_2$ this is no-longer the case, and the solution obtained with the nuclear norm generally have different singular vectors.
Since there is no closed form solution for this case we present a simple numerical evaluation, comparing \eqref{eq:nuclearobj} and \eqref{eq:muregprob} in Figure~\ref{fig:nncomp}.
The data was generated in the following way: First we constructed a rank 4 matrix $X$ of size $20 \times 20$ by selecting random matrices $U$ and $V$ of size $20 \times 4$ with i.i.d Gaussian entries with standard deviation $1$.
We then randomly selected an operator $\A$ represented by a $300 \times 400$ with i.i.d. Gaussian elements with standard deviation $\frac{1}{\sqrt{300}}$. It is known that operators of this type fulfills RIP with large probability \cite{recht-etal-siam-2010} (at least asymptotically). We then created the observation vector using $b=\A X + \epsilon$ where $\epsilon$ is Gaussian with standard deviation $s$. For the graphs in Figure~\ref{fig:nncomp} we used $s=0.1$ and $s=0.5$ to illustrate the effects of noise on the performance of the two methods. To circumvent issues with selecting optimal regularization weights (i.e.~$\lambda$ and $\mu$) for the two formulations we instead tested a range of values and plotted the resulting rank versus the data fit of the obtained solutions. It is clear that \eqref{eq:muregprob} gives better data fit for all ranks then \eqref{eq:nuclearobj}.
The difference between the two methods is larger when the noise level is larger due to the fact that the nuclear norm has to suppress a larger magnitude of the singular values to remove the noise. An interesting observation is that  \eqref{eq:muregprob} gives the same data fit regardless of $\mu$ as long as rank is the same. (Note that all curves contain 100 data points, however for \eqref{eq:muregprob} only one for each rank is visible since the rest are identical.) In contrast, to achieve the best possible performance with \eqref{eq:nuclearprox} $\lambda$ needs to be selected as small as possible while still yielding the correct rank. From a practical point of view it is preferable not to have to search for this value.

\subsection{Oracle type solutions for matrix recovery}\label{sec oracle}

Now assume that $b$ is of the form $\A X_0+\ep$ where $\ep$ is noise and $X_0$ is a matrix which has low rank $K$, which we will refer to as ``ground truth''. In the vector counterpart to the problems considered here, one usually has a sparse  vector $x_0$ (i.e.~a vector with few non-zero entries) whose support is known, and the best possible solution in this scenario is the so called ``oracle solution'' $x_{oracle}$, obtained by solving the equation system while imposing that the solution is zero outside of the known support. Note that we generally have $x_0\neq x_{oracle}$ due to the noise $\ep$. In \cite{carlsson2020unbiased} we proved that the global minimizer to the vector counterpart of the problems considered here equals $x_{oracle}$, given that certain assumptions are fulfilled (see e.g.~Corollary 2.3 of \cite{carlsson2020unbiased}). For the matrix case, it is not clear what the oracle solution should be. For example, in \cite{candes2011tight} it is suggested that the oracle solution $X_S$ be the one that you get if the ``oracle'' tells you the range of $X_0$, and you find $X_S$ by solving the linear problem
$$\A X_S=b,\quad \text{Ran} X_S\subseteq\text{Ran} X_0,$$
in the sense that \(X_S\) solves \( \min_X \|\A X - b\|^2_2 \) with the constraint \(\text{Ran}X_S\subseteq\text{Ran} X_0 \).
However, usually the range of $X_0$ has no particular meaning and in terms of data-fit this solution is suboptimal when compared to the ``best rank $K$ solution'', i.e.~the solution to the non-linear problem \eqref{best}, assuming only knowledge of $K$. The key message of this paper is that the methods proposed here have a high chance of finding this solution under appropriate assumptions on the noise level, matrix $\A$ and structure of $\sigma(X_0)$.

We now provide a simple example showing that some conditions are necessary, for otherwise it can even happen that the best rank $K$ solution does not exist. Consider the $2\times 2$ case with $K=1$, set $\A(X)=(x_{12},x_{21},x_{22})$ and $b=(1,1,0)$ so that $\A(X)=b$ becomes the equation system $x_{12}=x_{21}=1$ and $x_{22}=0$. Then $$X_k=\left(
                                                  \begin{array}{cc}
                                                    k & 1 \\
                                                    1 & 1/k \\
                                                  \end{array}
                                                \right) =
\left(                                                  \begin{array}{c}
k \\
1 \\
\end{array}
\right)
\left(                                                  \begin{array}{cc}
1&
1/k \end{array}
\right).
$$
is of rank 1 clearly satisfies $\|\A X_k-b\|_2\rightarrow 0$, but no rank 1 matrix can satisfy $\|\A X-b\|_2=0$.

A similar problem appears in this case for \eqref{eq:mupenalty} and its regularization \eqref{eq:muregprob} with $\mu=1$, say. With $\A$ and $b$ as above, the global minimum in both cases is then easily seen to be one, which is never attained. However, in this case we also have \begin{equation}\label{zero}\underset{\text{rank} X\leq 1}{\inf}\frac{\|\A X\|_2}{\|X\|_F}=0,\end{equation}
which is indicative of an ill-posed problem. In a sense it is an indication that we do not have enough measurements for determining a unique rank 1 solution. Another way of putting it is that then the RIP constant $\delta_1$ then is 1 or larger, and it is well known that even minimizing \eqref{eq:nuclearobj} does not necessarily yield a unique solution in this case.
In the next section we introduce theoretical conditions which rule out cases as the above one.

\subsection{Restrictions on $\A$ and the LRIP-condition}\label{sec:restrictiononA}

It is notoriously hard to determine if a problem instance has ``good'' RIP-values (cfr.~\eqref{eq:matrisRIP}) for a concrete application of fixed dimension. Classic compressed sensing theorems like those in \cite{CandesTaoRomberg2006} assume knowledge of the magnitude of \( \delta_K \) that can however be estimated only in probabilistic models (i.e.~where the measurements matrix \(A\) is random, drawn for instance with Gaussian or Sub-Gaussian distribution) and only in asymptotic scenarios, i.e.~when some parameter related to the dimension approaches infinity. Nontrivial deterministic models with prescribed RIP-values are very hard to construct (cfr.~\cite{Bourgain2011}). We refer to \cite{carlsson2020unbiased}, Section 2.3, for a more elaborate discussion on this.

For the theory developed in this paper we only need the lower estimate in \eqref{eq:matrisRIP}. We thus introduce the ``lower restricted isometry constant'' (LRIP) $\delta_K^-$ by setting $$1-\delta_K=\inf\left\{\frac{\|\A X\|^2_2}{\|X\|_F^2}:~ X\neq 0,~\text{rank}(X)\leq K\right\},$$ and we say that $\A$ satisfies an LRIP-condition of order $K$ if $\delta_K^{-}<1$ (c.f. \eqref{zero}). Clearly $\delta_K^-\leq \delta_K$ with equality whenever the lower bound in \eqref{eq:matrisRIP} is achieved. For some of our stronger theorems we will also use the assumption $\|\A\|\leq 1$, which clearly implies that $\delta_K^-=\delta_K$. Note however that we can have $\delta_K^-<0$, in contrast with $\delta_K$ which is automatically non-negative. 

As a remark, we note that LRIP-constants can be introduced for any operator $\A$ on any vector space, if the condition $\text{rank}(X)\leq K$ is swapped for some other condition. For example, if $\A$ is a matrix, $X$ is a vector and the condition is that $\text{card}(X)\leq K$, then we retrieve the LRIP-constants that were used in \cite{carlsson2020unbiased} and previously studied in \cite{blanchard2011compressed}. Thus the acronym LRIP can refer to different things depending on the context, just like RIP which was first invented for vectors and then modified in \cite{recht-etal-siam-2010} for matrices.


Simply assuming that $\delta_{K}^-<1$ gives that a best rank $K$ solution (c.f.~\eqref{best}) exists, although it may still be a set. To see this, suppose that  \((X_n)_{n=1}^\infty\) is a sequence of matrices with rank \( \le K \) such that $$ \lim_{n\rightarrow \infty}\|\A X_n-b\|_2 = \underset{\text{rank} X\leq K}{\inf}\|\A X-b\|_2 .$$ If $\| X_n \|_F$ converges to $ \infty $ then $ (1-\delta_K^-) \|X_n\|_F^2 \le \|\mathcal{A}(X_n)\|_2^2 \to \infty $ which would imply $ \|A(X_n) - b\|_2 \to \infty ,$ a contradiction.  It follows that \((X_n)_{n=1}^\infty\) is bounded and hence, by a standard compactness argument, that we can extract a subsequence which converges to a solution of \eqref{best}.
In the coming material we shall find that if $\delta_{2K}^-$ is sufficiently close to 0, then \eqref{best} has a unique solution.

\subsection{Key contributions}\label{sec key}

We first discuss our results in the concrete case of minimizing \eqref{eq:muregprob}. As an example of the type of results we provide, we have:

\begin{theorem}\label{glassintro}
Suppose that $b=\A X_0+\ep$ where $\|\A\|<1$ and $\emph{rank} (X_0)=K.$ Assume that
\[\|\ep\|_2\leq  \frac{ (1-\delta_{2K}^-)^{3/2}\sqrt{\mu}}{ 3}     \] and $$\sigma_K(X_0)>\para{\frac{1}{1-\delta_{2K}^-}+{(1-\delta_{2K}^-)}}\sqrt{\mu}.$$ Then the best rank $K$ solution $X_B$ is unique and equals the (also unique) global minimum to \eqref{eq:muregprob} as well as \eqref{eq:mupenalty}. Moreover \begin{equation}\label{gtf}\|X_B-X_0\|_F\leq 2\|\ep\|_2/\sqrt{1-\delta_{2K}^-}\end{equation} and $\emph{rank}(X)> K$ for any other stationary point $X$ of \eqref{eq:muregprob}.
\end{theorem}
We note that one can find stationary points of \eqref{eq:muregprob} using algorithms such as Forward-Backward Splitting (FBS), which under mild conditions is proven to converge to a stationary point, see Section \ref{sec num} for details. We also provide similar theorems on the uniqueness of low rank stationary points, see e.g. Theorem \ref{thm:statpoint:matrix}.

Note that the assumptions of the theorem are very natural; If the noise is too large or if the $K$:th singular value of $X_0$ is very small, there is clearly no chance of recovering $X_0$. Moreover the chance of recovering $X_0$ clearly relies on an appropriate choice of $\mu$, although the method is forgiving as long as $\mu$ is in the appropriate range, see Figure \ref{fig:nncomp}.

We have found no result in the literature which is as strong as this one. The results concerning nuclear norm minimization \eqref{eq:nuclearobj} also have estimates of the form \eqref{gtf}, but are suboptimal due to the shrinking bias and moreover usually include stronger assumptions, such as $\delta_{4K}$ being small. Another strength of the above result is that the constant in the estimate \eqref{gtf} grows slowly with $\delta_{2K}^-$; for example the value $\delta_{2K}^-=3/4$ gives the constant $4$, whereas similar estimates found in the literature usually apply only for much smaller values. For example, the key result of the celebrated paper \cite{recht-etal-siam-2010} applies when $\delta_{5K}<1/10$. 

Recent contributions concerning non-convex bilinear parametrization in the case when the model order $K $ is known, such as \cite{ge2017no,zhu2018global}, guarantee perfect recovery \textit{in the case of no noise.} To our knowledge, this is the first paper which gives conditions under which the best rank $K$ solution $X_B$ is a point of convergence for a low rank recovery method in the presence of noise.

In fact, in the noise free case we can prove a simpler result as follows
\begin{theorem}\label{thmnonoise}
Suppose that $b=\A X_0$ where $\emph{rank} (X_0)=K.$ Assume that $0<\delta_{2K}^-<1$ and that
$$\sigma_K(X_0)>\frac{\sqrt{\mu}}{1-\delta_{2K}^-}.$$ Then $X_0$ is a stationary point of \eqref{eq:muregprob} which is a unique rank $K$ minimizer, i.e. it solves $$\{X_0\}=\argmin_{\emph{rank}(X)\leq K} \sum_i r_\mu(\sigma_i(X)) + \|\A X-b\|_2^2.
$$ Moreover any other stationary point must have a higher rank.
\end{theorem}

In other words, the theorem says that there are no spurious rank $K$ local minima of our functional. This should be compared with recent influential contributions such as \cite{ge2017no} and \cite{zhu2018global}. Both papers display very promising informal versions of their results in the introduction, but a closer reading reveals that they only apply in the noise free setting. For example, Section 3.1 of \cite{ge2017no} considers precisely the situation discussed here with $\epsilon=0$, and concludes with a theorem guaranteeing that the method has $X_0$ as a stationary point if the RIP-constant $\delta_{2K}$ is less than $1/20$, which is hard to satisfy in practice. In contrast the above theorem applies if $\A$ satisfies
$$\sigma_K(X_0)>\frac{\sqrt{\mu}}{1-\delta_{2K}},$$ since $\delta_{2K}^-\leq \delta_{2K}$, as noted earlier. Similarly, Section III.C.1 of \cite{zhu2018global} considers minimization of $\|\A(X-X_0)\|_2$ and give recovery guarantees based on assuming $\delta_{4K}<1/5$. The main result of that paper does not seem to cover the noisy case, i.e.~when \(b=\A X_0 + \epsilon \) for $\epsilon\neq 0$, since it is assumed that the global minimum of the functional to be minimized already has low rank, which is very unlikely to hold when indeed $\epsilon\neq 0$.

Theorem \ref{thmnonoise} is a corollary of the results in Section \ref{sec:uniqnesspenalty}, and is simpler to prove than Theorem \ref{glassintro}, which is proved in Section \ref{noisy1}. However, the main contribution of this paper is more general than an analysis of the particular functional \eqref{eq:muregprob}. Suppose that $f$ is any sparsity inducing functional on $\R^n$, and that we form $F$ on the matrix space $\M_{n_1,n_2}$ by setting \begin{equation}\label{type}F(X)=f(\sigma(X)),\end{equation} where $n=\min\{n_1,n_2\}$.
A key theoretical contribution is Section \ref{sec:vectomatrix} which gives results on how to lift results concerning sparse vector estimation using $f$ to analogous results for low rank matrix estimation using $F$. Once this machinery is in place, theorems as those above follows ``easily'' by applying the methods developed in \cite{carlsson2020unbiased}.

To underline this point, we also investigate the concrete choice where $F$ is the quadratic envelope (see Section \ref{regS}) of the indicator functional $\iota_{R_K}$ used in \eqref{eq:fixedrank}. This functional is relevant when the sought rank $K$ is known a priori. In this case we prove a theorem similar to Theorem \ref{glassintro} (see Theorem \ref{t7}), but which is stronger in the sense that we do no need to find a suitable value of some parameter, such as $\mu$.

The paper is organized as follows, in Section \ref{regS} we briefly recall properties of the quadratic envelope, used to regularize discontinuous penalties such as $\mu\text{rank}(X)$ and $\iota_{R_K}$. In Section \ref{sec:unique} we recall some general observations about uniqueness of sparse stationary points. The paper really gets going in Section \ref{sec:vectomatrix} which provides the tools to lift results about vectors to matrices for penalties of the form \eqref{type}. We then consider the concrete cases of $F(X)=\mu\text{rank}(X)$ in Section \ref{secrank} and $F(X)=\iota_{R_K}(X)$ in Section \ref{secrankf}. In Section \ref{sec num} we discuss algorithms, primarily FBS and ADMM, and we conclude with some numerical examples indicating that our proposed estimator is unbiased, as the title claims.

\section{Relaxation via the Quadratic Envelope}\label{regS}
For easier reading we drop sub-indices on the norms so that e.g.~$\|\A X-b\|$, $\|X\|$ and $\|\A\|$ denotes $\ell^2$-, Frobenius- and operator-norms respectively.

Let $f$ be any lower semi-continuous (l.s.c.)~$[0,\infty]-$valued functional on $\M_{n_1,n_2}$ or more generally any scalar product vector space $\V$.
The quadratic envelope $\Q_\gamma(f)$, where $\gamma>0$ is a parameter, is designed such that $\Q_{\gamma}(f)(X)+\frac{\gamma}{2}\|X\|^2$ is the l.s.c. convex envelope of $f(X)+\frac{\gamma}{2}\|X\|^2$. A more direct way to introduce it is via the formula $\Q_\gamma=\S_\gamma\circ \S_\gamma$ where $\S_{\gamma}(f)(Y)=\sup_X - f(X)-\frac{\gamma}{2}\|{X-Y}\|^2$, we refer to  \cite{carlsson2018convex} for more details and further properties.
That paper also explores its potential use for relaxing problems of the form $f(X)+\frac{1}{2}\|\A X-b\|^2$ by replacing these with
\begin{equation}
\Q_\gamma(f)(X)+\frac{1}{2}\|\A X-b\|^2,
\label{rep}
\end{equation}
demonstrating several favorable properties. In this paper we will fix $\gamma=2$ and remove the traditional factor $\frac{1}{2}$ in front of the $\ell^2$-term, since it simplifies formulas. This is not a limitation since one can always obtain such a problem by rescaling $f,~ \A$ and $b$. Indeed, some simple computations easily yield $\Q_\gamma(f)=\frac{\gamma}{2}\Q_2(\frac{2}{\gamma}f)$ so that \eqref{rep} is equivalent to $$\Q_2\left(\frac{2}{\gamma}f \right)(x)+\left\|\frac{\A x}{\sqrt{\gamma}}-\frac{b}{\sqrt{\gamma}}\right\|^2.$$

We henceforth assume that such a rescaling has been done so that we are interested in minimizing
\begin{equation}\label{relax}\mathcal{R}_{reg}(X)=\Q_2(f)(X)+\left\|\A X-b\right\|^2\end{equation} instead of $\mathcal{R}(X)=f(X)+\|\A X-b\|^2$.
Reformulated to this setting, the main result of \cite{carlsson2018convex} reads as follows:

\begin{theorem}\label{celo1}
Let $\|\A\|<1$. If $x_l$ is a local minimizer (resp.~strict local minimizer) of $\mathcal{R}_{reg}$, then it is also a local minimizer (resp.~strict local minimizer) of $\mathcal{R}$, and $\mathcal{R}(x_l)=\mathcal{R}_{reg}(x_l)$. In particular, the sets of global minimizers of $\mathcal{R}$ and $\mathcal{R}_{reg}$ coincides.
\end{theorem}

\section{Properties of stationary points under LRIP}\label{sec:unique}

The results presented here follow closely those of Section 3 in \cite{carlsson2020unbiased}, but are included for completeness. They will in later sections be useful for establishing uniqueness of stationary points.
We say that a point $X$ is stationary for a functional $g$ if $$\underset{Y\neq X}{\liminf_{Y\rightarrow X}}~ \frac{g(Y)-g(X)}{\|Y\|}\geq 0.$$
For the case when $g$ is a sum of a convex function $g_c$ and a differentiable function $g_d$, we have that $X$ is a stationary point if and only if $$-\nabla g_d(X)\in \partial g_c(X),$$ where $\nabla g_d$ denotes the standard gradient and $\partial g_c(X)$ denotes the subdifferential commonly used in convex analysis. Setting
\begin{equation}
\G(X)=\frac{1}{2}\Q_2(f)(X)+\frac{1}{2}\|X\|^2,
\label{gdef}
\end{equation}
(so that $2\G$ is the l.s.c. convex envelope of $f(X)+\|X\|^2$), we have that
\begin{equation}\label{bat}
\mathcal{R}_{reg}(x)=2\G(X)-\|X\|^2+\|\A X-b\|^2,
\end{equation}
which upon differentiation yields that $X$ is a stationary point of $\mathcal{R}_{reg}$ if and only if
\begin{equation*}(I-\A^* \A)X + \A^* b\in\partial \G(X).\end{equation*}
Given any $X$, we therefore associate with it another point $Z$ defined by
\begin{equation}\label{y}
Z=(I-\A^* \A)X + \A^* b.
\end{equation}
Summing up, we have that $X$ is a stationary point of $\mathcal{R}_{reg}$ if and only if $Z$, defined via \eqref{y}, satisfies $Z\in\partial \G(X)$.
Moreover, it is useful to note that a stationary point $X$ of $\mathcal{R}_{reg}$ solves
\begin{equation}\label{hy}
\min_{Y} \Q_{2}(f)(Y) + \|Y-Z\|^2,
\end{equation}
i.e. it minimizes the l.s.c. convex envelope of $f(Y)+\|Y-Z\|^2$, which was shown in Proposition 3.2 of \cite{carlsson2020unbiased}.

The following result on uniqueness of sparse stationary points of $\mathcal{R}_{reg}$ (defined by \eqref{relax} for any non-negative penalty $f$) is taken from Section 3 of \cite{carlsson2020unbiased}, and is included for completeness.

\begin{proposition}\label{thm:statpoint}
Let $X'$ and $X''$ be a stationary points of $\mathcal{R}_{reg}$ with $\emph{rank}(X')+\emph{rank}(X'')\leq 2K$. We then have that
\begin{equation}\label{gd1}
 \re \scal{Z''-Z',X''-X'}\leq  \delta_{2K}^-\|X''-X'\|^2,
\end{equation}
where $Z''$ and $Z'$ are defined via \eqref{y}.
\end{proposition}
\begin{proof}
Note that $Z''-Z'=(I-\A^*\A)(X''-X')$. Taking a scalar product with $X''-X'$ and noting that $\text{rank}(X''-X')\leq 2K$ gives $$\re\langle Z''-Z',X''-X'\rangle=\|X''-X'\|^2-\|\A(X''-X')\|^2\leq (1-(1-\delta_{2K}^-))\|X''-X'\|^2,$$
as desired.\end{proof}

The above proposition is not very interesting in itself, but the point is the following: Suppose we have found one stationary point $X'$ with $\emph{rank}(X')\leq K$, and suppose that we can show that the reverse inequality to \eqref{gd1} holds for all $X''$ with $\rank(X'')\leq K$ and $Z''\in \partial\mathcal{G}(X'')$. Then it follows by contradiction that $X'$ is a unique stationary point with rank less than or equal to $K$. 

An even stronger result would be to say that that $X'$ is a unique solution to \begin{equation}\label{k}\argmin_{\emph{rank}(X)\leq K}\mathcal{R}_{reg}(X),\end{equation} but this does not follow from the above since the minimizers under the extra condition $\emph{rank}(X)\leq K$ are not necessarily stationary points. Hence it is not even clear if $X'$ is a solution. The following proposition fills this gap upon assuming a bit more.

\begin{proposition}\label{thm:statpoint1}
Let $X'$ be a stationary point of $\mathcal{R}_{reg}$ with $\emph{rank}(X')\leq K$. Suppose that \begin{equation}
 \re \scal{W-Z',Y-X'}>  \delta_{2K}^-\|Y-X'\|^2,\label{ineq}
\end{equation}
 holds for all $Y$ with $\emph{rank}(Y)\leq 2K$ and $W\in\partial \G(Y) $, then $X'$ is the solution to \eqref{k} and as such unique.
\end{proposition}


\begin{proof}
Let $X''\neq X'$ be a solution to \eqref{k}, and let $Y$ be any point on the line between $X'$ and $X''$. It follows that $\rank(Y)\leq 2K$.
By \eqref{bat} we have that $\mathcal{R}_{reg}(Y)$ is a combination of the convex term $2\mathcal{G}(Y)$ and a smooth term $-\|Y\|^2+\|\mathcal{A}Y-b\|^2$, hence its directional derivative in any given direction $V$ exists and equals $$(\mathcal{R}_{reg})'_V(Y)=2\re \big(\sup_{W\in \partial\mathcal{G}(Y)}\langle W-Y,V\rangle+\langle\A^*(\A Y-b),V\rangle\big).$$  By \eqref{y} we have $$0=\langle Z'-X',V\rangle+\langle\A^*(\A X'-b),V\rangle$$
which, upon subtracting in the previous equation, gives $$(\mathcal{R}_{reg})'_V(Y)\geq 2\re \big(\langle W-Z',V\rangle-\langle Y-X',V\rangle+\langle\A^*\A (Y-X'),V\rangle\big)$$
for any $W\in \partial\mathcal{G}(Y)$. Inserting $V=Y-X'$ and using that \eqref{ineq} holds, we conclude that $$(\mathcal{R}_{reg})'_{Y-X'}(Y)> 2 \big(\delta_{2K}^- \|Y-X'\|^2-\|Y-X'\|^2 +\|\A (Y-X')\|^2\big)\geq 0,$$
where the last inequality follows from the definition of $\delta_{2K}^-$. It follows that the function $t\mapsto \mathcal{R}_{reg}(X'+t(X''-X'))$ has a positive right derivative at every point $t>0$, and hence $\mathcal{R}_{reg}(X'')> \mathcal{R}_{reg}(X')$. This shows that $X'$ is a solution to \eqref{k} and moreover that as such it is unique, as desired.
\end{proof}

A central ingredient in our previous work \cite{carlsson2020unbiased} is to prove statements such as \eqref{ineq} in the vector setting. In the coming section we provide tools to lift such statements to the matrix setting, which will enable us to prove results similar to those in \cite{carlsson2020unbiased} for matrices instead of vectors. This is done in Sections \ref{secrank} and \ref{secrankf}.


\section{Lifting results about vectors to matrices}\label{sec:vectomatrix}

Let us now say that we are interested in finding low rank matrices $X$ in  $\M_{n_1,n_2}$, and set $n=\min(n_1,n_2)$. We let $X=U\Lambda_{\sigma(X)} V^*$ be the singular value decomposition of $X$ where the vector of singular values is denoted $\sigma(X)$. In the case when $n_1\neq n_2$ we use the definition where $U\in\M_{n_1,n_1}$ and $V\in\M_{n_2,n_2}$
are unitary, $\sigma(X)\in \R^n$ and $\Lambda_{\sigma(X)}$ denotes a non-square diagonal matrix (with $\sigma(X)$ on the diagonal).

If $f$ is a sparsity inducing functional on $\R_n$, then it is natural that $X\mapsto f(\sigma(X))$ is a low rank inducing functional on $\M_{n_1,n_2}$. An example of this is the nuclear norm, which arises as $\|X\|_{*}=\|\sigma(X)\|_1$ or even $\text{rank}(X)$, which equals $\text{card}(\sigma(X))$.
Although this is often straightforward to implement, it is unfortunately not trivial to ``lift'' results about vectors to the matrix setting. In this section we provide several results simplifying such ``liftings'', with a particular focus on quadratic envelopes and Proposition \ref{thm:statpoint1}.

For example, inequalities such as \eqref{ineq} are established in \cite{carlsson2020unbiased} for vectors, and it is natural to hope that they also apply in the matrix setting. This is indeed the case, but the matrix problem  is substantially more difficult since the effects of the unitary matrices $U$, $V$ from the SVD cannot be ignored. In what follows we will show that tightness of the bounds occur when $U$ and $V$ are matrices
that permute and change signs of vector elements. Therefore lifting to the matrix setting can be done by considering the worst case permutations and sign changes of the singular values.


A functional $f:\R^n \to \mathbb{R}$ is called absolutely symmetric if $f(x)=f(\Pi x)$ for all $\Pi\in Per$ and \begin{equation}\label{extf}f(x)=f\big((|x_1|,\ldots,|x_n|)\big), \quad x\in\R^n.\end{equation}
Given any such $f$, we extend it to $\M_{n_1,n_2}$ by setting
\begin{equation}\label{liftS}F(X)=f(\sigma(X)),\quad X\in\M_{n_1,n_2}.\end{equation}
The following results show how to connect the theory for $F$ with a scalar theory for $f$.
Given a vector $\gamma$ we let $\Lambda_\gamma$ denote the corresponding diagonal matrix with $\gamma$ on the diagonal. We omit the details of the following basic proof.

\begin{lemma}\label{l0g}
If $f$ is absolutely symmetric and $\Pi\in  Per$ is a permutation then $z\in\partial f(x)$ if and only if $\Pi z\in\partial f(\Pi x)$. If $\gamma$ is a vector with only $\pm 1$, then $\partial f(\Lambda_{\gamma} x)=\Lambda_{\gamma}\partial f(x)$.
\end{lemma}

At this point we need the definitions of $\Q_2$ and $\S_2$, which were given in Section~\ref{regS}.

\begin{lemma}\label{l1g}
If $f$ is absolutely symmetric the same holds for $\S_2(f)$ and $\Q_2(f)$.
\end{lemma}
\begin{proof}
Since $\Q_2=\S_2\circ\S_2$ it suffices to prove that $\S_2(f)$ is absolutely symmetric. It is easy to see that $f$ is absolutely symmetric if and only if $f(\Lambda_\gamma\Pi x)=f(x)$, for any $\gamma$ and $\Pi$ as in the previous lemma. Since both $\Pi$ and $\Lambda_{\gamma}$ are unitary and the Frobenius norm is invariant under multiplication by unitary matrices, we have \begin{align*}&\S_{2}(f)(\Lambda_\gamma\Pi y)=\sup_x - f(x)-\|{x-\Lambda_\gamma\Pi y}\|^2=   \\&\sup_x - f((\Lambda_\gamma\Pi )^*x)-\|{(\Lambda_\gamma\Pi )^*x-y}\|^2=\S_2(f)(y).\end{align*}
\end{proof}

\begin{proposition}\label{l2g}
Let $f$ be absolutely symmetric and  let $F$ be given by \eqref{liftS}. Then $\Q_2(F)(X)=\Q_{2}(f)(\sigma(X))$.
\end{proposition}
\begin{proof}
$\S_2(F)(Y)$ is given by $$\sup_X - f(\sigma(X))-\|{X-Y}\|^2=\sup_X - f(\sigma(X))-\|X\|^2+\|Y\|^2-2\Re{\scal{X,Y}}.$$ von Neumann's trace inequality implies that the supremum is attained for an $X$ that shares singular vectors with $Y$ (see e.g.~\cite{carlsson2021neumann}). For such $X$ we have $\scal{X,Y}=\sum_{j=1}^n\xi_j\sigma_j(Y)$ where $\xi$ is a reordering of the singular values of $X$. Since $f$ is symmetric we have $f(\sigma(X))=f(\xi)$ and so
$$\S_{2}(F)(Y)=\sup_{\xi\in\R_+^n} - f(\xi)-\|\xi-\sigma(Y)\|^2=\sup_{\xi\in\R^n} - f(\xi)-\|\xi-\sigma(Y)\|^2=\S_2(f)(\sigma(Y))$$ where \eqref{extf} was used in the second identity.
The corresponding identity for $\Q_2$ follows by iterating this twice; $$\Q_2(F)(X)=\S_2(\S_2(F))(X)=\S_2(\S_2(f))(\sigma(X))=\Q_2(f)(\sigma(X)).$$
\end{proof}

Recall the function $\G$ introduced in \eqref{gdef}, which is central to our argument.  In the present setting, we have both $\G_f$ (for vectors) and $\G_F$ (for matrices). Our next task is to relate these two.

\begin{lemma}\label{l3g}
With $f, ~F, ~\G_f$ and $\G_F$ as above, we have $$\G_F(X)=\G_f(\sigma(X)).$$ Moreover, if $U$ and $V$ are unitary then $U\partial\G_F(X)=\partial\G_F(U^*X)$ and $\partial\G_F(X)V=\partial\G_F(XV^*)$. Finally, given SVD $X=U\Lambda_{\sigma(X)}V^*$, we have $$\partial\G_F(X)=U\Lambda_{\partial\G_f(\sigma(X))}V^*.$$
\end{lemma}
\begin{proof}
By the identity $\|X\|=\|\sigma(X)\|$ and Proposition \ref{l2g} we have $$2\G_F(X)=\Q_2(F)(X)+\|X\|^2=\Q_2(f)(\sigma(X))+\|\sigma(X)\|^2=2\G_f(\sigma(X)).$$ The following two identities are easily derived using the invariance of the Frobenius scalar product under multiplication by unitary matrices, i.e.~$\scal{A,B}=\scal{UA,UB}$ and $\scal{A,B}=\scal{AV,BV}$. The last identity now follows by Corollary 2.5 in \cite{lewis1995convex}.
\end{proof}
To use Propositions \ref{thm:statpoint} and \ref{thm:statpoint1} we are interested in the quantity \begin{equation}\label{g14}
\underset{Y\neq X'}{\inf_{Y\in\M_{n_1,n_2}}}\inf_{W\in\partial \G_F(Y)}\frac{\re\scal{W-Z',Y-X'}}{\|Y-X'\|^2}.
\end{equation}
The main difficulty lies in reduction to the scalar case, which we now show how to do. This requires some preparation. We refer the reader to \cite{andersson2016operator}, Section 3, for the basics of complex doubly substochastic (CDSS) matrices. In particular we need that the set of CDSS matrices is convex with extreme points of the form $\Lambda_\gamma\Pi$, where $\gamma$ is a vector with unimodular complex entries, and $\Pi\in Per$. Before the proof we introduce some notation, the symbol $\R^n_{\geq}$ refers to all non-increasing sequences of $\R^n$ and $O_n$ will denote the set of all unitary matrices.

\begin{proposition}\label{propQuotg}
Let \(P \subseteq \mathbb{R}^n \) be a set which is sign and permutation invariant. Let $X'=U_{X'}\Lambda_{x'}V_{X'}^*$ be a singular value decomposition of $X'$ where $x'\in \R^n_\geq$. Fix $Z'\in\partial \G_F(X')$ and let $z'\in \R^n_\geq$ be its vector of singular values. Then $$\underset{\sigma(Y) \in P}{\underset{Y\neq X'}{\inf_{Y\in \M_{n_1,n_2}}}} \inf_{W\in\partial \G_F(Y)}\frac{\re\scal{W-Z',Y-X'}}{\|Y-X'\|^2}=\underset{y\neq x'}{\inf_{y\in P}}\inf_{w\in\partial \G_f(y)}\frac{\scal{w-z',y-x'}}{\|y-x'\|^2}.$$
Moreover, if one infimum is attained, then so is the other.
\end{proposition}
\begin{proof}
For notational simplicity we assume that $n_1=n_2=n$ and remove the obvious $Y\neq X'$ and $y\neq x'$ from the below expressions, since the general case $n_1\neq n_2$ follows by straightforward modifications. By Lemma \ref{l0g} we have that \[ \begin{split}
& \inf_{y\in P}\inf_{w\in\partial \G_f(y)}\frac{\scal{w-z',y-x'}}{\|y-x'\|^2}  \\  = & \inf_{y\in\R^n_\geq \cap P}\inf_{w\in\partial \G_f(y)}\underset{\gamma\in\{-1,1\}^n}{\min_{\Pi\in Per}}\frac{\scal{\Lambda_\gamma\Pi w-z',\Lambda_\gamma\Pi y-x'}}{\|\Lambda_\gamma\Pi y-x'\|^2}
\end{split} \]
On the other hand, by Lemma \ref{l3g} and the fact that the Frobenius scalar product is invariant under multiplication by unitary matrices, it follows that we can take $U_{X'}=V_{X'}=I$ and we get \[ \begin{split}
&\underset{\sigma(Y) \in P}{\inf_{Y\in\M_{n_1,n_2}}}\inf_{W\in\partial \G_F(Y)}\frac{\re\scal{W-Z',Y-X'}}{\|Y-X'\|^2}\\=& \underset{\sigma(Y)\in P}{ \inf_{Y\in\M_{n_1,n_2}}}\inf_{W\in\partial \G_F(Y)}\frac{\re\scal{W-\Lambda_{z'},Y-\Lambda_{x'}}}{\|Y-\Lambda_{x'}\|^2}\\=&\inf_{y\in\R^n_\geq \cap P}\inf_{w\in\partial \G_f(y)}\min_{U,V\in O_n}\frac{\re\scal{U\Lambda_w V^*-\Lambda_{z'},U\Lambda_y V^*-\Lambda_{x'}}}{\|U\Lambda_y V^*-\Lambda_{x'}\|^2}
\end{split} \]
where we use the fact that $O_n$ is compact so we can be sure that the above minimum is attained. By comparing the two expressions we see that it suffices to show
\begin{equation}\label{ht5g}\begin{aligned}
&\min_{U,V\in O_n}\frac{\re\scal{U\Lambda_w V^*-\Lambda_{z'},U\Lambda_y V^*-\Lambda_{x'}}}{\|U\Lambda_y V^*-\Lambda_{x'}\|^2}\\=&\min_{\gamma\in\{-1,1\}^n}{\min_{\Pi\in Per}}\frac{\scal{\Lambda_\gamma\Pi w-z',\Lambda_\gamma\Pi y-x'}}{\|\Lambda_\gamma\Pi y-x'\|^2}\end{aligned}
\end{equation}
for fixed $y\in\R^n_{\geq}$ and $w\in\partial \G_f(y)$. Let us denote the minimum on the lower line by $c$. Note that $\Pi\Lambda_y\Pi^*=\Lambda_{\Pi y}$ and that $\Pi^*=\Pi^{-1}\in O_n$ for all $\Pi\in Per$. Thus, replacing $U$ by $\Lambda_\gamma\Pi $ and $V$ by $\Pi$ we get \begin{align*}&\min_{U,V\in O_n}\frac{\re\scal{U\Lambda_w V^*-\Lambda_{z'},U\Lambda_y V^*-\Lambda_{x'}}}{\|U\Lambda_y V^*-\Lambda_{x'}\|^2}\\ \leq & \min_{\gamma\in\{-1,1\}^n}{\min_{\Pi\in Per}}\frac{\scal{\Lambda_\gamma\Pi\Lambda_w \Pi^*-\Lambda_{z'},\Lambda_\gamma\Pi\Lambda_y \Pi^*-\Lambda_{x'}}}{\|\Lambda_\gamma\Pi\Lambda_y \Pi^*-\Lambda_{x'}\|^2}\\=&\min_{\gamma\in\{-1,1\}^n}\min_{\Pi\in Per}\frac{\scal{\Lambda_\gamma\Lambda_{\Pi  w}-\Lambda_{z'},\Lambda_\gamma\Lambda_{\Pi  y}-\Lambda_{x'}}}{\|\Lambda_\gamma\Lambda_{\Pi  y}-\Lambda_{x'}\|^2}\\=&\min_{\gamma\in\{-1,1\}^n}\min_{\Pi\in Per}\frac{\scal{\Lambda_\gamma\Pi w-z',\Lambda_\gamma\Pi y-x'}}{\|\Lambda_\gamma\Pi y-x'\|^2}=c\end{align*}
so to establish \eqref{ht5g} we just need to prove the reverse inequality. This is more difficult, we shall show the equivalent inequality \begin{equation}\label{3g}\min_{U,V\in O_n}{\re\scal{U\Lambda_w V^*-\Lambda_{z'},U\Lambda_y V^*-\Lambda_{x'}}}-c{\|U\Lambda_y V^*-\Lambda_{x'}\|^2}\geq 0.\end{equation} Treating $y,x',w,z'$ as column vectors and using $\odot$ for Hadamard multiplication of matrices, we first note that \[\begin{split}
&{\re\scal{U\Lambda_w V^*-\Lambda_{z'},U\Lambda_y V^*-\Lambda_{x'}}}-c{\|U\Lambda_y V^*-\Lambda_{x'}\|^2}\\=&{\re\scal{U\Lambda_w -\Lambda_{z'}V,U\Lambda_y-\Lambda_{x'}V}}-c{\|U\Lambda_y-\Lambda_{x'}V\|^2}\\= &
{c_1-\Re{\scal{U\Lambda_w,\Lambda_{x'}V}+\overline{\scal{U\Lambda_y,\Lambda_{z'}V}}}}+2c\re{\scal{U\Lambda_y,\Lambda_{x'}V}}\\= &
{c_1-\Re{{x'}^t( U \odot\bar V)w+{z'}^t(\bar U \odot V)y}}+{2c\Re{{x'}^t( U \odot \bar V)y}}\end{split}\]
where $c_1$ is a constant (i.e.~independent of $U$ and $V$). Then we note that $ U \odot \bar V$ is a complex doubly sub-stochastic matrix. Since the function \begin{equation*}B\mapsto {c_1-\Re{{x'}^tBz+{z'}^t \bar B y}}+2c\Re{{x'}^t B y}\end{equation*} is affine, it will attain its minimum (over the convex set of complex doubly sub-stochastic matrices) in an extreme point. By the comments before the proof, we conclude that there exists a vector of unimodular entries $\gamma$ and $\Pi\in Per$ such that the minimum equals $${c_1-\Re{{x'}^t\Lambda_\gamma \Pi w+{z'}^t \Lambda_{\bar \gamma} \Pi y}}+2c\Re{{x'}^t \Lambda_\gamma \Pi y}.$$ Clearly the unimodular numbers in $\gamma$ have to be either $+1$ or $-1$ in order for a minimum to be reached. By following the above computations backwards this can be written $${\re\scal{ \Lambda_\gamma \Pi w-z',\Lambda_\gamma \Pi x'-{x'}}}-c{\|\Lambda_\gamma \Pi y-{x'}\|^2}$$
which by the definition of $c$ clearly is greater or equal to 0. This establishes \eqref{3g} and the proof is complete.
\end{proof}

\section{Matrix case, $F=\mu\text{rank}$}\label{secrank}

In the coming two sections we consider two concrete penalties with the aim of lifting the results about sparse vector estimation from \cite{carlsson2020unbiased} to the case of matrices. In this section we want to minimize \begin{equation}\label{t571}\mathcal{R}_{\mu} (X):=\mu \, \text{rank}(X)+\|\A X-b\|^2\end{equation} which we replace by \begin{equation}\label{fr21}\mathcal{R}_{\mu,reg}(X):=\Q_2(\mu \, \text{rank})(X)+\|\A X-b\|^2.\end{equation}

To connect these expressions with the previous sections, we set $f(x)=\mu~ \text{card}(x)$ and define $F$ via \eqref{liftS}, i.e. $F(X)=f(\sigma(X))$, which yields $F(X)=\mu \, \text{rank}(X)$. Proposition \ref{l2g} then gives that
\begin{equation}\label{l00001}\Q_2(\mu \, \text{rank})(X)=\Q_2(\mu \, \text{card})(\sigma(X))\end{equation} and therefore we can reuse various results from Section 4 of \cite{carlsson2020unbiased} (which considers $\Q_2(\mu \, \text{card})$). The above formula makes it clear why the expression \eqref{eq:murankconvenv} is the l.s.c.~convex
envelope of \eqref{eq:murankapprox}, since a minor computation shows that  $\Q_2(\mu\text{card})(x)=\sum_j r_\mu(x_j)$ (see e.g.~\cite{carlsson2020unbiased}). Hence we see that \eqref{fr21}
is just another way of writing \eqref{eq:muregprob}. Finally, by Theorem \ref{celo1} we know that $\mathcal{R}_{\mu,reg}$ has the same global minima as $\mathcal{R}_\mu$, as well as potentially fewer local minima, as long as $\|\A\|<1$.

\subsection{On the uniqueness of sparse stationary points} \label{sec:uniqnesspenalty}

Given $K$ such that $\delta_{2K}^-<1$, we will show that under certain assumptions the difference between two stationary points always has rank larger than $2K$.
Hence, if we find a stationary point with rank less than $K$, then we can be sure that this has the smallest rank among the stationary points. The main theorem reads as follows:

\begin{theorem}\label{thm:statpoint:matrix}
Let $X'$ be a stationary point of $\mathcal{R}_{\mu,reg}$ with $\emph{rank}(X')\leq K$, let $Z'$ be given by \eqref{y}, and assume that
\begin{equation}\label{e41}\sigma_i(Z') \not\in\left[{(1-\delta_{2K}^-)}{\sqrt{\mu}},\frac{1}{1-\delta_{2K}^-}{\sqrt{\mu}}\right].\end{equation}
Then $X'$ is the unique solution to $$\argmin_{\emph{rank}(X)\leq K}\mathcal{R}_{\mu,reg}(X)$$ and moreover $\emph{rank}( X'') > K$ for any other stationary point of $\mathcal{R}_{\mu,reg}$. Finally, if $\A$ satisfies $\|\A\|<1$ and \begin{equation}\label{cond222}\|\A X'-b\|^2\leq \mu,\end{equation} then $X'$ is a global minimum of both $\mathcal{R}_{\mu}$ and $\mathcal{R}_{\mu,reg}$.
\end{theorem}

To gain some intuition about the point $Z'$ we recall that by \eqref{hy} we have that $X'$ solves
\begin{equation}
\min_X \Q_{2}(\mu \text{rank})(X) + \|X-Z'\|^2.
\end{equation}
By the discussion following \eqref{eq:nuclearprox} we have that $X'$ can be computed from $Z'$ by performing an SVD of $Z'$ and hard threshold the singular values at $\sqrt{\mu}$, (while keeping the singular vectors unchanged).
Loosely speaking the theorem says that if the singular values of $Z'$ are not too close to the threshold $\sqrt{\mu}$, then the difference to any other stationary point has to be of high rank. Whether this is true or not depends on the level of noise, which we study in Section~\ref{noisy1}.

In particular, Theorem \ref{thmnonoise} in the introduction deals with the noise free case, and it is easy to see that it follows from Theorem \ref{thm:statpoint:matrix} and some simple computations;
In this case we have $b = \A X_0$ for some $X_0$ with rank $K$, and we select $\mu$ so that $\frac{\sqrt{\mu}}{1-\delta_{2K}^-}$ is smaller than $\sigma_K(X_0)$, where it is assumed that $0<\delta_{2K}^-<1$. This implies that the non-zero values of $\sigma(X_0)$ are larger than $\sqrt{\mu}$ which, considering the concrete expression for $r_{\mu}$, implies that $X_0$ is a local minimizer of \eqref{eq:muregprob}, and hence it is a stationary point. Moreover we have $Z_0 = (I-\A^*\A )X_0 - \A^* \A X_0 = X_0$, which clearly fulfills the assumptions of the above theorem, by which the remaining assertions in Theorem \ref{thmnonoise} follow.

For the proof of Theorem~\ref{thm:statpoint:matrix} we make use of Section \ref{sec:vectomatrix} applied to the functionals \( \mathcal{G}_{\mathcal{Q}_2(\mu \text{card})  }   \) on $\R^n$ and \( \mathcal{G}_{\mathcal{Q}_2(\mu\text{rank})  }   \) on $\M_{n_1,n_2}$. We will also need Lemmas 4.3 and 4.4 of \cite{carlsson2020unbiased} which can be summed up as follows.

 \begin{lemma}\label{corbor}
Let $z'$ satisfy $z'\in\partial \mathcal{G}_{\mathcal{Q}_2(\mu\emph{card})}(x')$ and
	\begin{equation}\label{e4t}z_i'\not\in\left[(1-\delta_{2K}^-){\sqrt{\mu}},\frac{1}{1-\delta_{2K}^-}{\sqrt{\mu}}\right],\end{equation}
where we assume that $0<\delta_{2K}^-<1$.	If $y\neq x'$ and $w\in\partial \mathcal{G}_{\mathcal{Q}_2(\mu\emph{card})}(y)$, then $\scal{w-z',y-x'}>\delta_{2K}^-\|y-x'\|^2$.
\end{lemma}

\begin{proof}[Proof of Theorem \ref{thm:statpoint:matrix}]
The first two statements follows immediately by Proposition \ref{thm:statpoint} and \ref{thm:statpoint1} if we show that \begin{equation*}
\re \skal{W-Z',Y-X'} > \delta_{2K}^- \|Y-X'\|^2.
\end{equation*} for any $Y\neq X'$ and $W\in\partial\mathcal{G}_{\mathcal{Q}_2(\mu \text{rank})}(Y)$. Suppose the converse. By the results of Section \ref{sec:vectomatrix} we have $\mathcal{G}_{\mathcal{Q}_2(\mu \text{rank})} (Y)=\mathcal{G}_{\mathcal{Q}_2(\mu\text{card})}(\sigma(Y))$ and thus Proposition \ref{propQuotg} implies that there are (real) vectors $y\neq \sigma(X')$ and $w$ with $w \in \partial \mathcal{G}_{\mathcal{Q}_2(\mu \text{card})}(y)$ such that
\begin{equation}\label{gr}  \skal{w-z',y-x'} \leq \delta_{2K}^- \|y-x'\|^2
\end{equation}
for  $x'=\sigma(X'),~z'=\sigma(Z') \in\partial \mathcal{G}_{\mathcal{Q}_2(\mu\emph{card})}(x')$, where the last inclusion follows by Lemma \ref{l3g}. This contradicts Lemma \ref{corbor} in the case when $0<\delta_{2K}^-<1$.

We now consider the remaining case \( \delta_{2K}^-\leq 0  \). As noted in Section \ref{sec:unique} the function $\mathcal{G}_{\Q_2(\mu \text{card})}$ is l.s.c. convex and it is well known that it follows that its subdifferential is monotone, i.e.~that
$\langle w-z',y-x'\rangle\geq 0.$ When \( \delta_{2K}^-< 0  \) this again contradicts \eqref{gr}.

It remains to consider the case \( \delta_{2K}^-=0 \). For this we use some ideas from Theorem 4.2 of \cite{carlsson2020unbiased}, we briefly outline the details. There is a function $g$ such that $ \mathcal{G}_{\mathcal{Q}_2(\mu \text{card})}(x)=\sum_{j=1}^n g(x_j)$ (see equation (29)-(30) of \cite{carlsson2020unbiased}). Then \eqref{gr} implies that there is a $y\neq x'$ and $w\in\partial \mathcal{G}_{\mathcal{Q}_2(\mu \text{card})}(x)$ such that $\skal{w-z',y-x'} \leq 0$. By assumption we have that \( z'_i \ne  \sqrt{\mu}  \) for all \( 1 \le i \le n  \) which implies \( x'_i \notin (0,\sqrt{\mu}]  \) for all \( 1 \le i \le n  \), by inspection of the graph of $\partial g$ (see (31) in \cite{carlsson2020unbiased}). There must exist at least one index $i$ such that $y_i\neq x_i'$. It follows that $w_i\neq z_i'$ and that $(w_i-z_i')(x_i-x_i')>0$, again by inspection of the graph of $\partial g$. Thus $\scal{z-z',x-x'}>0$, which is a contradiction. This finishes the proof of the first two claims.

The last affirmation follows by utilizing Theorem \ref{celo1} and straightforward adaption of the argument of Theorem 4.5 in \cite{carlsson2020unbiased} to the present situation. We omit the details.

\end{proof}

\subsection{Noisy data.}\label{noisy1}

We now come to one of the main results of the paper, which already was mentioned in the introduction (Theorem \ref{glassintro}). It should be compared with Corollary 2.1 of \cite{carlsson2020unbiased}, treating the corresponding problem for vectors. The result is basically the same, albeit with less sharp constants. The proof on the other hand is quite different, since we can not rely on explicit formulas for the oracle solution in the present setting. The role that is played by the oracle solution in the vector setting will be replaced by the best rank $K$ solution to $\A X=b$, i.e.~the solution $X_B$ to \eqref{best}.

\begin{theorem}\label{glass}
Suppose that $b=\A X_0+\ep$ where $\|\A\|<1$ and $\emph{rank} (X_0)=K.$  Assume that $\mu$ satisfies
\[\|\ep\|\leq  \frac{{(1-\delta_{2K}^-)}^{3/2}\sqrt{\mu}}{ 3}     \] and $$\sigma_K(X_0)>\para{\frac{1}{{1-\delta_{2K}^-}}+(1-\delta_{2K}^-)}\sqrt{\mu}.$$ Then \eqref{best} has a unique solution $X_B$ which equals the unique global minimum to $\mathcal{R}_{\mu,reg}$ as well as $\mathcal{R}_{\mu}$. Moreover $$\|X_B-X_0\|\leq 2\|\ep\|/\sqrt{1-\delta_{2K}^-}$$ and $\emph{rank}(X'')> K$ for any other stationary point $X''$ of $\mathcal{R}_{\mu,reg}$.
\end{theorem}

\begin{proof}
In the noise free case the theorem is true by Theorem \ref{thmnonoise}, so we may assume that $\ep\neq 0$.
As noted earlier we have $\Q_2(\mu \text{rank})(X)=\sum_j r_{\mu}(\sigma_j(X))$ and hence it follows that $$\mu \text{rank}(X)=\Q_2(\mu \text{rank})(X)$$ as long as the non-zero singular values of $X$ are all larger than $\sqrt{\mu}$. Since this is clearly true for $X_0$, we have $$\mathcal{R}_{\mu}(X_0)=\mathcal{R}_{\mu,reg}(X_0)=\mu K+\|\ep\|^2$$ so the global minimum must be smaller than this value. If $\text{rank}(X)>K$ we therefore have $\mathcal{R}_{\mu}(X)>\mathcal{R}_{\mu}(X_0)$ in view of $\mu>\|\ep\|^2$, so the global minimizer of $\mathcal{R}_{\mu}$ must have rank $\leq K$. By Theorem \ref{celo1} we know that $\mathcal{R}_{\mu}$ and $\mathcal{R}_{\mu,reg}$ share global minimizers, so this implies that a global minimizer $X'$ to $\mathcal{R}_{\mu,reg}$ must satisfy $\text{rank}(X')\leq K$. Note that a global minimizer indeed exists, since $\delta_K^{-}\leq\delta_{2K}^-<1$ and $$\mathcal{R}_{\mu,reg}(X)= \Q_2(\mu \text{rank} (X))+\|\A X-b\|^2\geq (1-\delta_{K}^-) \|X\|^2-2\re\scal{\A X,b}+\|b\|^2$$ for matrices $X$ with $\text{rank}(X)\leq K$, so a sequence $(X_k)_{k=1}^\infty$ such that $\mathcal{R}_{\mu,reg}(X_k)$ converges to the global minimum must be bounded, and the desired conclusion is immediate by the compactness of bounded sets in finite dimensional metric vector spaces (and continuity of $\mathcal{R}_{\mu,reg}$). By the same token we have that the set of best rank $K$ solutions is non-empty, as noted in Section \ref{sec:restrictiononA}.

Now assume that $X'$ is a global minimizer and that $\text{rank} (X')=K-L$ with $L\geq 1$. Given fixed singular values of $X'$, recall that $\|X'-X_0\|^2$ attains its minimum when $X'$ share singular vectors with $X_0$ \cite{carlsson2021neumann}, which easily gives that $$\|X'-X_0\|^2\geq \sum_{j=K-L+1}^K \sigma_j^2(X_0)\geq L \sigma_K^2(X_0).$$ Note that the a priori estimate for \( \|\epsilon\| \) and for \( \sigma_K (X_0)  \) can be combined to give 
\begin{align*}&\|\A X'-b\|=\|\A (X'-X_0)-\epsilon\|\geq \|{\A (X'-X_0)}\|-\|\epsilon\|\\&\geq \sqrt{1-\delta_{2K}^-}\| X'-X_0\|-\|\epsilon\|\geq \sqrt{1-\delta_{2K}^-}\sqrt{L} \sigma_K(X_0)-\|\epsilon\| \\& \geq\sqrt{L}\para{\frac{1}{\sqrt{1-\delta_{2K}^-}}+{(1-\delta_{2K}^-)}^{3/2}}\sqrt{\mu}-\|\epsilon\|\geq\sqrt{L}\sqrt{\mu}+3\|\epsilon\|-\|\epsilon\|\\&\geq \sqrt{L\mu}+2\|\ep\|\end{align*} where we used the fact that $\sqrt{1-\delta_{2K}^-}\leq \|\A\|<1$. We get \begin{equation*}\begin{split}&\mathcal{R}_{\mu,reg}(X') =\mathcal{R}_{\mu}(X') =\mu (K-L)+\|\A X'-b\|^2\ge  \\ &   \mu (K-L)+(\sqrt{L\mu}+2\|\ep\|)^2 > \mu K+4\|\ep\|^2=\mathcal{R}_{\mu,reg}(X_0)\end{split}\end{equation*} where we used that $\ep\neq 0$.
This shows that $\text{rank}(X')<K$ also is impossible, so we conclude that $\text{rank}(X')=K$ for any global minimizer $X'$.

Since all global minimizers of $\mathcal{R}_{\mu}$ have rank equal to $K$, it follows that any global minimizer $X'$ is a best rank $K$ solution as defined in \eqref{best}, and vice versa that any solution to \eqref{best} is a global minimizer of $\mathcal{R}_{\mu}$. With this at hand, we have \begin{equation*}\begin{split}&\|\ep\|^2  =\mathcal{R}_{\mu}(X_0)-K\mu\geq\mathcal{R}_{\mu}(X')-K\mu=\|\A X'-b\|^2\geq\\&  (\|\A (X'-X_0)\|-\|\ep\|)^2\geq (\sqrt{1-\delta_{2K}^-}\|X'-X_0\|-\|\ep\|)^2 \end{split}\end{equation*} and by the earlier computations (see the above estimation of $\|\A X'-b\|$) we know that $\sqrt{1-\delta_{2K}^-}\|X'-X_0\|-\|\ep\|\geq 0$. Taking the square root and rearranging gives $\|X'-X_0\|\leq 2\|\ep\|/\sqrt{1-\delta_{2K}^-}$, which is the final estimate in the theorem.

We now address the uniqueness of global minimizers. Let $X'$ be a global minimizer, which thus has rank $K.$ If $\mathcal{R}_{\mu}$ would have multiple global minimizers, this would imply the existence of another rank $K$ stationary point of $\mathcal{R}_{\mu,reg}$. To conclude the uniqueness part of the proof, it thus suffices to show that $\text{rank} (X'')>K$ for any stationary point $X''$ of $\mathcal{R}_{\mu,reg}$ other than $X'$.


This follows if we show that Theorem \ref{thm:statpoint:matrix} applies, and indeed all remaining affirmations then follow as well. Let $Z'$ be given by \eqref{y} and recall that $Z'\in\partial \G_{\Q_2(\mu\text{rank})}(X')$, which by Lemma \ref{l3g} implies that \begin{equation}\label{tgf}\sigma(Z')\in\partial \G_{\Q_2(\mu\text{card})}(\sigma(X')).\end{equation}
We need to prove that \eqref{e41} applies. To this end we first recall that $|\sigma_j(X')-\sigma_j(X_0)|\leq \|X'-X_0\|$ (which follows e.g.~by the Hoffman-Wielandt inequality). Thus \begin{equation*}\begin{split}&\sigma_K(X')  \ge \sigma_K(X_0)-\|X'-X_0\| \ge \para{\frac{1}{1-\delta_{2K}^-}+(1-\delta_{2K}^-)}\sqrt{\mu}-\frac{2}{\sqrt{1-\delta_{2K}^-}}\|\ep\|\\ & \ge \para{\frac{1}{1-\delta_{2K}^-}+(1-\delta_{2K}^-)-\frac{2}{3}(1-\delta_{2K}^-)}\sqrt{\mu} > \frac{1}{1-\delta_{2K}^-}\sqrt{\mu}.\end{split}\end{equation*}
{We have $\G_{\Q_2(\mu\text{card})}=\frac{1}{2}(\sum_{j=1}^n r_{\mu}(x_j)+x_j^2)$ where $r_{\mu}(x_j)$ is constant for $x_j>\sqrt{\mu}$.}
By \eqref{tgf} it follows that $\sigma_j(Z')=\sigma_j(X')$ whenever $\sigma_j(X')\geq \sqrt{\mu}$. This proves that $\sigma_j(Z')>\frac{1}{1-\delta_{2K}^-}{\sqrt{\mu}}$ for all $j\leq K$, as desired. Finally, \begin{align*}Z'=( I-\A^* \A)X' + \A^* b = X'-\A^*\A(X'-X_0)-\A^*\ep \end{align*}
so for $j>K$ we have \begin{equation*}\begin{split}&|\sigma_j(Z')| =|\sigma_j(Z')-\sigma_j(X')| \le \|Z'-X'\|\le  \|\A^*\|\|\A\|\|(X'-X_0)\|+\|\A\|\|\ep\|
 \le \\ &\|X'-X_0\|+\|\ep\| \leq \left( \frac{2}{\sqrt{1-\delta_{2K}^-}} + 1 \right) \|\epsilon \|\leq \left( \frac{3}{\sqrt{1-\delta_{2K}^-}}  \right) \|\epsilon \| \end{split}\end{equation*}
which is less than $(1-\delta_{2K}^-) \sqrt\mu$ under the assumption on the noise's magnitude. This establishes \eqref{e41}, so by Theorem \ref{thm:statpoint:matrix} it follows that all stationary points $X''$ of $\mathcal{R}_{\mu,reg}$ other than  $X'$ satisfy $\text{rank}(X''-X')>2K$, which means that $\text{rank}(X'')>K$, as desired.

\end{proof}


\section{Matrix case, fixed rank}\label{secrankf}

The second concrete example we wish to investigate in this paper is the choice $f=\iota_{R_K}$ (which was defined in \eqref{iotaRk}). In this case, the unregularized problem \eqref{eq:fixedrank} coincides with the problem of finding the best rank $K$ solution $X_B$ as defined in \eqref{best}. The main difference is that we now assume the model order $K$ to be known; one problem instance where this is true is for example the PhaseLift approach to the phase retrieval problem (see \cite{Carlssonphase2019}). This method does not require a correct parameter choice $\mu$ in order to find $X_B$, and hence it is the method of choice whenever $K$ is explicitly known.

To be more precise, set $f=\iota_{K}$ defined via
\begin{equation}
\iota_{K}(x) =
\begin{cases}
0 & \text{card}(x)\leq K \\
\infty & \text{otherwise}
\end{cases}\label{iotak}
\end{equation}
 which by \eqref{liftS} gives $F(X)=\iota_{R_K}(X)=\iota_K(\sigma(X))$  and
\begin{equation}\label{l00001f}\Q_2 ( \iota_{R_K})(X) = \Q_2 ( \iota_K)(\sigma(X))
\end{equation} by Proposition \ref{l2g}. As before, we want to minimize \begin{equation}\label{t571f}\mathcal{R}_K (X):=\iota_{R_K}(X) +\|\A X-b\|^2\end{equation} which we replace by \begin{equation}\label{fr21f}\mathcal{R}_{K,reg} (X):=\Q_2 ( \iota_{R_K})(X)+\|\A X-b\|^2,\end{equation} where the latter has the same global minima and potentially fewer local minima, as long as $\|\mathcal{A}\|<1$ (Theorem \ref{celo1}). The key point is that minimizing the regularized version has a much higher chance of actually finding $X_B$.

\subsection{On the uniqueness of sparse stationary points}\label{sec:uniqnessfixed}
The following result gives a condition for uniqueness of sparse (i.e. rank $\leq K$) stationary points.
\begin{theorem}\label{thm:statpoint:matrixfix}
Let $X'$ be a stationary point of $\mathcal{R}_{K,reg}$ with  $\emph{rank}(X') \le K$, let $Z'$ be given by \eqref{y}, and assume that
\begin{equation}\label{e4fix567}\sigma_{K+1}( Z')<(1 - 2 \delta_{2K}^{-})\sigma_{K}(Z').\end{equation}
Then there are no other stationary points with rank less than or equal to $K$.
\end{theorem}
As in the previous section we need a result from \cite{carlsson2020unbiased}, in this case the details are found in the proof of Theorem 5.2.
 \begin{lemma}\label{corbor1}
Let $z'$ satisfy $z' \in\partial \mathcal{G}_{\iota_{K}}(x')$ with \(\emph{card}(x') \le K \) and
\begin{equation*}\tilde z_{K+1}'<(1-2\delta_{2K}^-)\tilde z_{K}'.\end{equation*}
If $y \neq x'$, $w \in\partial \mathcal{G}_{\iota_{K}}(y)$ and \(\emph{card}(y) \le K \), then $$\scal{w-z',y-x'}>\delta_{2K}^-\|y-x'\|^2.$$
\end{lemma}

\begin{proof}[Proof of Theorem \ref{thm:statpoint:matrixfix}]
The result follows by Proposition \ref{thm:statpoint} if we show that \[ \langle W - Z', Y - X' \rangle > \delta_{2K}^{-} \|Y - X' \|^2;      \] holds for all $Y \in R_K$ with $Y\neq X'$ and $W \in\partial \G_{\iota_{R_K}}(Y)$.	If not then \[ \underset{Y \in R_K}{\underset{Y \ne X'}{\inf_{ Y \in \mathbb{M}_n }}} \inf_{W \in \partial \G_{\iota_{R_K}}(Y)}   \frac{\langle W - Z', Y - X' \rangle}{\|Y - X' \|^2  } \le \delta_{2K}^{-}     \] and thus Proposition \ref{propQuotg} yields that there exists $y \neq x'$ with $\text{card}(y)\leq K$ and $w \in \partial \G_{\iota_{K}}(y)$ such that \[ \frac{\scal{w-z',y-x'}}{\|y-x'\|^2} \le \delta_{2K}^{-} , \] which contradicts Lemma \ref{corbor1}.
\end{proof}

If we add the assumption $\|\A\|<1$ we can easily prove that the above sparse minimizer is the global minimizer as well.

\begin{theorem}\label{t6}
Suppose that \( \|\mathcal{A}\| < 1 \). Then there exists a global minimizer \(X'\) of \( \mathcal{R}_{K,reg}  \); if \( Z' \) given by \eqref{y} satisfies \eqref{e4fix567}, then \( X'  \) is {unique} and there are no other local minimizers either. 	
\end{theorem}

\begin{proof} That $\mathcal{R}_K$ has a minimizer is shown via a simple compactness argument using $\delta_{2K}^{-} \ge \delta_K^{-}$, which we did already in Section \ref{sec:restrictiononA}. By Theorem \ref{celo1} any such minimizer is also a minimizer of $\mathcal{R}_{K, reg}$, and vice versa.

Now assume that $Z'$ satisfies \eqref{e4fix567} as stipulated. If \(Y\) is another local minimizer, then \( \text{rank}(Y) > K \) by Theorem \ref{thm:statpoint:matrixfix}. But then we have \( \mathcal{R}_{K,reg}(Y)=\mathcal{R}_K(Y)=+\infty  \) by Theorem \ref{celo1}, a contradiction.
\end{proof}

\subsection{Noisy data.}

In this final section we consider the case when $b=\A(X_0)+\ep$ for a low rank $X_0$.

\begin{theorem}\label{t7}
	Suppose that $b=\A X_0+\ep$,  $\|\A\|<1$, \( \delta_{2K}^{-} < 1/2 \) and $\emph{rank} (X_0)=K$. Assume that
	\[ \sigma_K(X_0) > \left(  \frac{5}{ (1 - 2 \delta_{2K}^{-})^{3/2}  }    \right)\|\epsilon\|.      \]  Then the best rank $K$ solution $X_B$ is the unique global minimum to $\mathcal{R}_{K,reg}$ and there are no other local minimizers. Moreover \[ \|X_B-X_0\| \le 2\|\epsilon\| / \sqrt{1-\delta_{2K}^{-}}.  \]
\end{theorem}

We first need a lemma.
\begin{lemma}\label{ld}
	Given \( Y \in R_K \) we have \( W \in \partial \mathcal{G}_F (Y)  \) if and only if \( \sigma_j (W) = \sigma_j (Y)   \) for \(j =1, \dots , K   \).
\end{lemma}
\begin{proof} By Lemma \ref{l3g} we have \(W \in \partial \mathcal{G}_F(Y) \) if and only if \( \sigma(W) \in \partial \mathcal{G}_f (\sigma(Y))  \) and the conclusion follows from of Lemma 5.3 of \cite{carlsson2020unbiased}.
	
\end{proof}

\begin{proof}[Proof of Theorem \ref{t7}] The existence of a global minimum $X'$ follows by Theorem \ref{t6}, and due to the simple structure of $\mathcal{R}_{K}$ it is immediate that $X'=X_B$. The estimate on $\|X'-X_0\|$ follows by the simple computation: \[ \begin{split} &\|\epsilon\| = \sqrt{ \mathcal{R}_K (X_0)} \ge \sqrt{ \mathcal{R}_K (X')}  = \|\mathcal{A}(X'-X_0) - \epsilon\| \ge \\ &  | \|\mathcal{A}(X'-X_0)\| - \|\epsilon\| | \ge \|\mathcal{A}(X'-X_0)\| - \|\epsilon\| \ge \sqrt{1-\delta_{2K}^{-}} \|X'-X_0\| - \|\epsilon\|.  \end{split}  \] It remains to verify uniqueness, which follows by Theorem \ref{t6} once we prove that \eqref{e4fix567} applies to $Z'$ (given by \eqref{y}). First of all we notice that Hoffman-Wielandt inequality gives $ \|X'-X_0\| \ge |\sigma_K (X') - \sigma_K(X_0)| $ so $\sigma_K(X') > \sigma_K (X_0) - \frac{2}{\sqrt{1-\delta_{2K}^{-}}} \| \epsilon \|$ and therefore \[ \sigma_K(Z') > \sigma_K (X_0) - \frac{2}{\sqrt{1-\delta_{2K}^{-}}} \| \epsilon \|   \] by Lemma \ref{ld}. Moreover, the last lines of the proof of Theorem \ref{glass} gives an estimate for \( \sigma_{K+1}(Z')  \), i.e. \[\sigma_{K+1}(Z') \le \frac{3}{\sqrt{1-\delta_{2K}^{-}}} \|\epsilon\|,  \] (which applies in the present setting as well). The hypothesis \[ \sigma_K(X_0) > \left(  \frac{5}{  (1 - 2 \delta_{2K}^{-})^{3/2}  }    \right)\|\epsilon\|  \] combined with these two estimates gives \[ \frac{\sigma_{K+1}(Z')}{1 - 2\delta_{2K}^{-}} \le \frac{3\|\epsilon\|}{ (1 - 2 \delta_{2K}^{-})^{3/2}} < \frac{5\|\epsilon\|}{ (1 - 2 \delta_{2K}^{-})^{3/2}}-\frac{2\|\epsilon\|}{\sqrt{1-\delta_{2K}^{-}}} < \sigma_K(Z') \] which is \eqref{e4fix567}, and the proof is complete.

\end{proof}

\section{Minimization and proximal operators}\label{prox_alg}

Before getting to some numerical tests, let us discuss implementation issues. Both algorithms FBS (Forward-Backward Splitting) and ADMM (Alternating Directions Method of Multipliers) are capable of finding stationary points of \eqref{fr21} and \eqref{fr21f}, according to our numerical observations. It seems that the theory supporting this claim is more developed for the case of FBS. In \cite{attouch2013convergence} it is shown that FBS generates sequences that either diverge to $\infty$ or else converge to a stationary point, for semi-algebraic functionals. The functionals $\Q_2(\mu \text{rank})$ and $\Q_2(\iota_K)$ are semi-algebraic, which follows by Theorem 6.1 in \cite{carlsson2018convex} along with the fact that the singular values are semi-algebraic functions of the matrix entries. For the case of ADMM the theory in the non-convex case is more unclear; on one hand there are examples where ADMM diverges \cite{li2015global}, on the other the article \cite{wang2019global} gives conditions under which (some alteration of) ADMM converges to a stationary point. The latter also has a long list of concrete settings where ADMM has been reported to converge. We can only add to this list, we have never encountered a situation where ADMM diverges and moreover we have run extensive tests with ADMM and FBS on the same problem, observing that they seem to find the same point. A benefit with ADMM over FBS is that it allows also to incorporate linear constraints. In order to use either ADMM or FBS we need to be able to compute the proximal operators of $\Q_2(\mu\text{rank})$ and $\Q_2(\iota_{R_K})$ respectively. The details of how to do this is found e.g.~in \cite{larsson-olsson-ijcv-2016}, but we repeat it below for completeness. The code is also available in MatLab at the following GitHub repository:

\url{https://github.com/Marcus-Carlsson/Quadratic-Envelopes.}

We remind the reader that proximal operators are defined as $$\text{prox}_{\Q_2(F)/\rho}(X)=\argmin_{Y}\Q_2(F)(X)+\frac{\rho}{2}\|X-Y\|^2.$$ 	
First of all we provide a result to connect the proximal operator of \( \Q_{2}(f) \) to the proximal operator of \( \Q_{2}(F) \) with \( F(X)=f(\sigma(X)) \), being \( \sigma(X) \) the singular values vector of a matrix \( X \in \mathbb{M}_{n_1, n_2} \). The following theorem is taken from \cite{Carlssonphase2019}, (Proposition 2.1).
	
\begin{proposition} \label{liftprox} Let $f$ be a permutation and sign invariant $[0,\infty]$-valued function on $\mathbb{R}^n$ and set $F(X)=f(\sigma(X))$. Then for $\rho > 2$
		\begin{equation*}
		\emph{prox}_{\Q_2(F)/\rho}(X)=U \emph{diag}(\emph{prox}_{\Q_2(f)/\rho}(\sigma(X)))V^*
		\end{equation*}
		where $U \emph{diag}(\sigma(X)) V^*$ is the singular value decomposition of $X$.
	\end{proposition}

We give now formulas for \( \text{prox}_{\Q_{2} (\mu \text{card}) } \) and \( \text{prox}_{\Q_{2} (\iota_K) } \); the corresponding for \( \text{prox}_{\Q_{2} (\mu \text{rank}) } \) and \( \text{prox}_{\Q_{2} (\iota_{P_k}) } \) will follow from Proposition \ref{liftprox}:
\begin{equation} \label{proxcard}
(\text{prox}_{\Q_{2}(\mu \text{card} ) / \rho } (y))_i = \begin{cases} y_i & \text{if } |y_i| \ge \sqrt{\mu} \\ \frac{\rho y_i - 2 \sqrt{\mu} \text{sign}(y_i) }{\rho -2} & \text{if } 2 \sqrt{\mu}/\rho \le |y_i| \le \sqrt{\mu} \\ 0 & \text{if } |y_i| \le 2\sqrt{\mu} / \rho. \end{cases}
\end{equation} \( \text{prox}_{\Q_{2} (\iota_K) } \) is much more complicated to compute; the details are found in \cite{larsson-olsson-ijcv-2016} and \cite{Carlssonphase2019}. The  Read Me file of the GitHub repository also has a step by step description of the code that can be used for implementing in any other programming language.

We finally describe the Forward-Backward Splitting algorithm specialized to our two functionals:
	\newline \newline

\begin{algorithm}[H]
		\SetAlgoLined
		\KwResult{Stationary point of \( \mathcal{R}_{\mu,reg} \) or \( \mathcal{R}_{K,reg} \)}
		\(X_1 = 0\) (\(n_1 \times n_2\) zero matrix)\;
		\While{not converged}{
			\(\widehat{X}_{k+1} = X_k - 2 \A^* (\A X_k - b )/\rho\)\;
			\( \widehat{X}_{k+1} = U_{k+1} \text{diag}(\sigma(\widehat{X}_{k+1})) V^* _{k+1}  \) (singular value decomposition)\;
			\( X_{k+1} = U_{k+1} \text{diag}(\text{prox}_{ \Q_{2} (\mu \text{card})/\rho } (\sigma(\widehat{X}_{k+1}) ) V_{k+1} ^*  \)\;
			\textbf{or}\\
			\( X_{k+1} = U_{k+1} \text{diag}(\text{prox}_{ \Q_{2} (\iota_K)/\rho } (\sigma(\widehat{X}_{k+1}) ) V_{k+1} ^*  \)
		}
		\caption{Forward-Backward Splitting for \( \mathcal{R}_{\mu,reg} \) and \( \mathcal{R}_{K,reg} \)}
\end{algorithm}

\section{Numerical results}\label{sec num}
\begin{figure}
\begin{center}
\resizebox{60mm}{!}{
%
%
\definecolor{mycolor1}{rgb}{0.00000,0.44700,0.74100}%
\definecolor{mycolor2}{rgb}{0.85000,0.32500,0.09800}%
\begin{tikzpicture}

\begin{axis}[%
width=4.521in,
height=3.566in,
at={(0.758in,0.481in)},
xmin=0,
xmax=5,
ymin=0,
ymax=7,
axis background/.style={fill=white},
legend style={at={(0.05,0.803)}, anchor=south west, legend cell align=left, align=left, draw=white!15!black}
]
\addplot [color=mycolor1, line width=2.0pt]
  table[row sep=crcr]{%
0	2.01119474901381e-11\\
0.5	0.38178846111302\\
1	0.764760330386044\\
1.5	1.14547172676452\\
2	1.53460019512312\\
2.5	1.8934937406637\\
3	2.2917471245571\\
3.5	2.67241071635111\\
4	3.04643337570499\\
4.5	3.44241554198248\\
5	3.85791194511585\\
};
\addlegendentry{$\mathcal{Q}_2(\iota_{R_K})(X)$}

\addplot [color=mycolor2, line width=2.0pt]
  table[row sep=crcr]{%
0	0.103416239914905\\
0.5	1.08607417577159\\
1	1.9445680762641\\
1.5	2.49008418053233\\
2	3.1654571495393\\
2.5	3.8245382321735\\
3	4.4654597335831\\
3.5	5.15112591870583\\
4	5.66245446235411\\
4.5	6.18087882450405\\
5	6.94020204797098\\
};
\addlegendentry{$\|X\|_*$}

\end{axis}
\end{tikzpicture}%
}	
\resizebox{60mm}{!}{
%
%
\definecolor{mycolor1}{rgb}{0.00000,0.44700,0.74100}%
\definecolor{mycolor2}{rgb}{0.85000,0.32500,0.09800}%
\begin{tikzpicture}

\begin{axis}[%
width=4.521in,
height=3.566in,
at={(0.758in,0.481in)},
xmin=0,
xmax=5,
ymin=0,
ymax=9,
axis background/.style={fill=white},
legend style={at={(0.05,0.803)}, anchor=south west, legend cell align=left, align=left, draw=white!15!black}
]
\addplot [color=mycolor1, line width=2.0pt]
  table[row sep=crcr]{%
0	5.49810339347876e-11\\
0.5	0.424411739532788\\
1	0.843392174162233\\
1.5	1.26567338699154\\
2	1.69129109165286\\
2.5	2.15745033232739\\
3	2.51147282933911\\
3.5	2.94641845369635\\
4	3.44199238561985\\
4.5	3.88329021645677\\
5	4.30544742861236\\
};
\addlegendentry{$\mathcal{Q}_2(\iota_{R_K})(X)$}

\addplot [color=mycolor2, line width=2.0pt]
  table[row sep=crcr]{%
0	0.530013575597306\\
0.5	1.6909698276051\\
1	2.93927906972725\\
1.5	3.6186609115908\\
2	4.49212061669807\\
2.5	5.18984518976033\\
3	5.91874571490652\\
3.5	6.69460884632077\\
4	7.24401381515784\\
4.5	7.77982864132385\\
5	8.52186941367874\\
};
\addlegendentry{$\|X\|_*$}

\end{axis}
\end{tikzpicture}%
}	
\end{center}
\caption{Comparison of regularization with $\mathcal{Q}(\iota_{R_K})$ vs. $\|X\|_*$. \emph{Left:} Data fit $\|\A X - b\|^2$ vs. noise level $\|\epsilon\|$. \emph{Right:} Ground truth distance $\|X-X_0\|$ vs. noise level $\|\epsilon\|$.}
\label{fig:iotaexp}
\end{figure}
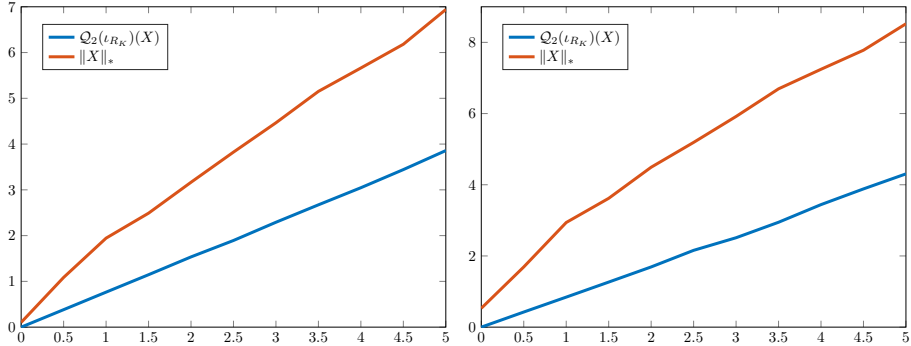

Figure~\ref{fig:iotaexp} shows the results of minimizing \eqref{fr21f} under varying levels of noise $\|\epsilon\|$. For comparison we also plot the results obtained when minimizing \eqref{eq:nuclearobj}. We use the same setup as in Section~\ref{sec:bias}, with a matrix $X_0$ of rank $4$ and size $20 \times 20$, and an operator $\A$ represented by a $300 \times 400$ matrix. The data was generated by $b = \A X_0 + \epsilon$ with varying $\|\epsilon\|$.
To ensure that the minimization of \eqref{eq:nuclearobj} gives the best possible data fit we search for the
smallest $\lambda$ giving the correct rank using a bisection strategy \cite{bazaraa-etal-2005}. The graphs in Figure~\ref{fig:iotaexp} shows the data fit and the distance to the ground solution $X_0$ for both methods.
\begin{figure}
	\begin{center}
\resizebox{100mm}{!}{\input{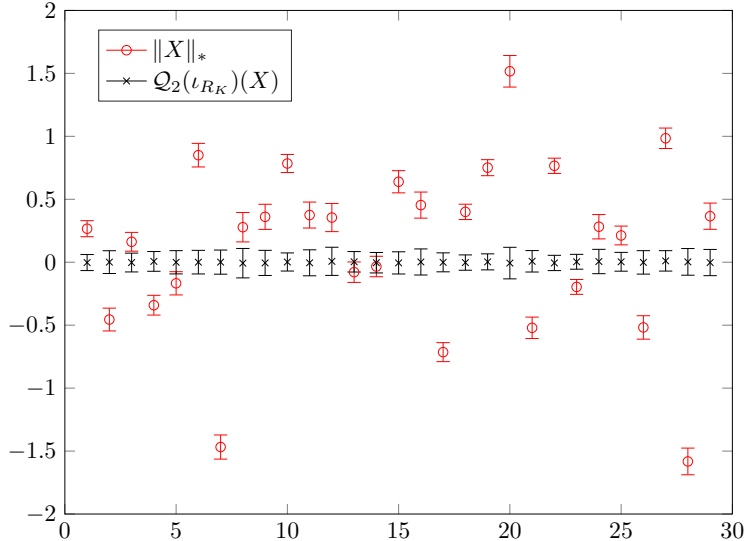}}	
\end{center}
\caption{Estimated means of $X-X_0$ ($\pm$ 2 standard deviations) for a few random elements in the estimated matrix. }
\label{fig:iotabias}
\end{figure}

In Figure~\ref{fig:iotabias} we estimated the bias of \eqref{eq:nuclearobj} and \eqref{fr21f}. With the same setup as above we generated $100$ instances where only the noise vector $\epsilon$ was varied. We used a fixed noise level $\|\epsilon\| = 1$. We then estimated means and standard deviations for the elements of $X-X_0$ from all the solutions using either \eqref{eq:nuclearobj} or \eqref{fr21f}.
The results plotted in Figure~\ref{fig:iotabias} clearly illustrate that under this noise model nuclear norm regularization gives a statistically biased estimation as opposed to $\mathcal{Q}_2(\iota_{R_K})$.

{\color{red}
	\subsection{Experiments with real data}
	In this section we test our relaxations on an application of Non-Rigid Structure from Motion (NRSfM).
	Given images of a deforming object, captured with a moving camera, the goal is to reconstruct a 3D model of the object. Typically the object is represented by a sparse 3D point cloud and we observe 2D projections of these 3D points. Under the assumption of a linear shape basis \cite{bregler-etal-cvpr-2000} the complexity of the deformation can be measured by the rank of a matrix containing the 3D point locations when the different images where captured. We will follow the setup in Dai et al. \cite{dai-etal-ijcv-2014}, where $X$ is a matrix where each row $i$ contains to the $x$-, $y$- and $z$-coordinates of the 3D points when image $i$ was captured. The matrix can be recovered using
	a formulation of the form
	\begin{equation}
	\min_X \mathcal{S}(X) + \|R X^\# - M\|^2.
	\end{equation}
	Here $X^\#$ is a reshaped version of $X$ where three consecutive rows contain the $x$-, $y$-, $z$-coodinates from an image. The matrix $R$ contains the camera rotations and $M$ contains the measured projections, see \cite{dai-etal-ijcv-2014} for further details. The term $\mathcal{S}(X)$ is some regularization of the rank of $X$. In our experiments we test the nuclear norm, $\Q_2(\mu \rank)$ and $\Q_2(\iota_{K})$.
	
	The linear operator $X \mapsto R X^\#$ does not obey the LRIP constraint since it has rank $1$ matrices in its nullspace. Hence there is no guarantee that our relaxations do not end up in local minima.
	Figure~\ref{fig:mocap} shows the results obtained for the four MOCAP sequences used in \cite{dai-etal-ijcv-2014}. We varied the regularization parameters $\lambda,$ $\mu$ and $K$ and plot the resulting rank (x-axis) versus the data fit (y-axis). It can be seen that our relaxations consistently achieve better data fits for all ranks despite potentially being susceptible to local minima. Hence even in cases where we cannot give theoretical guarantees it may still of interest to apply our non-convex relaxations because of their lack of shrinking bias.
	\begin{figure}[htb]
		\begin{center}
			\def\w{27mm}
			\begin{tabular}{cccc}
				Drink & Pickup & Stretch & Yoga \\
				\includegraphics[width=\w]{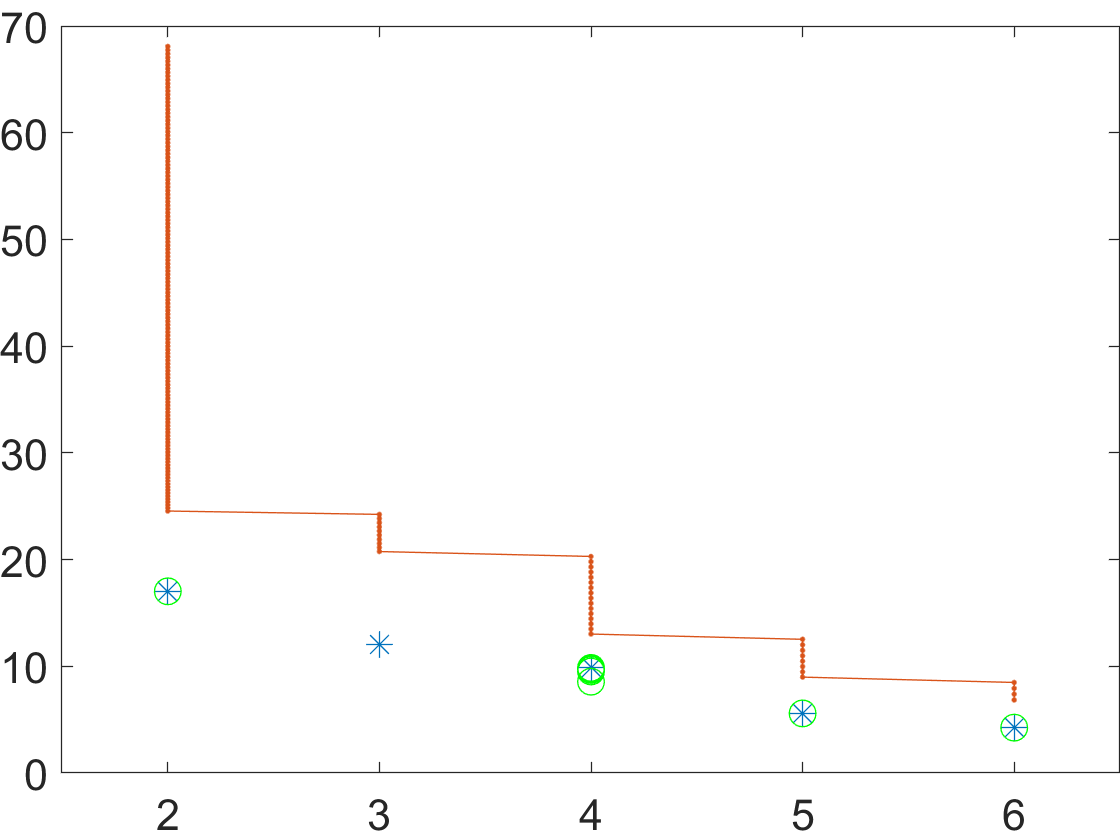} &
				\includegraphics[width=\w]{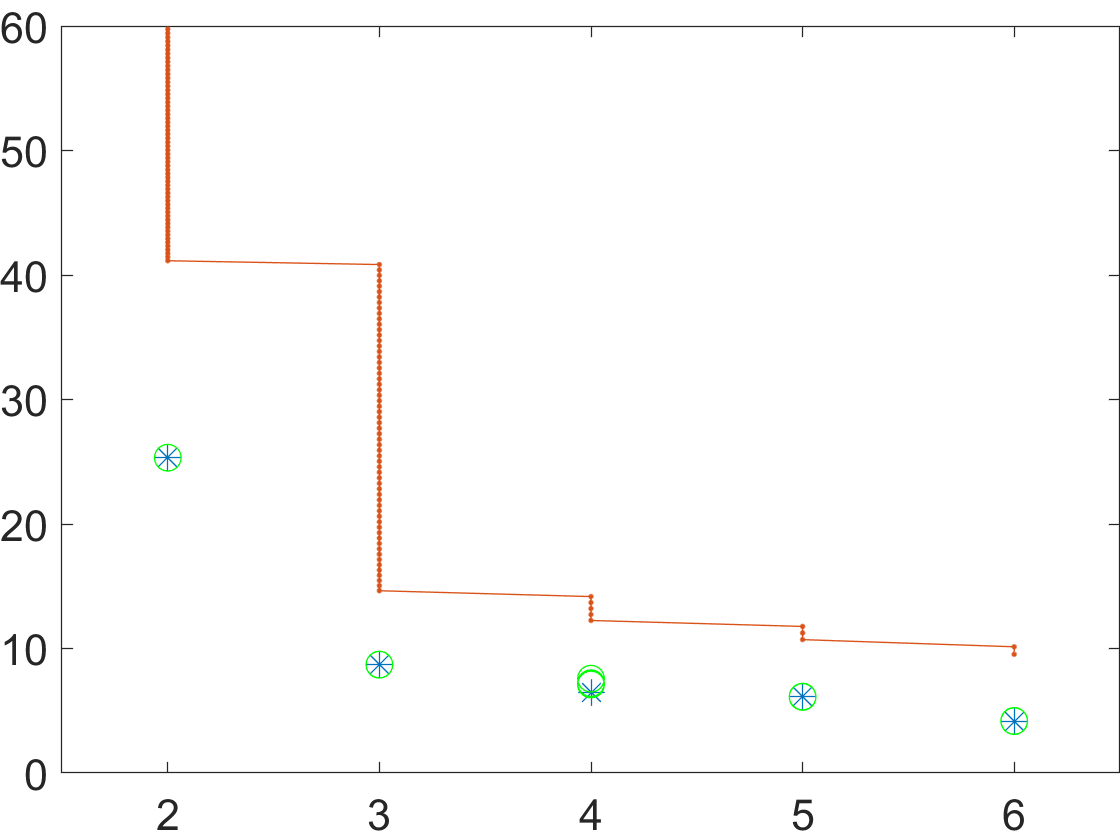} &
				\includegraphics[width=\w]{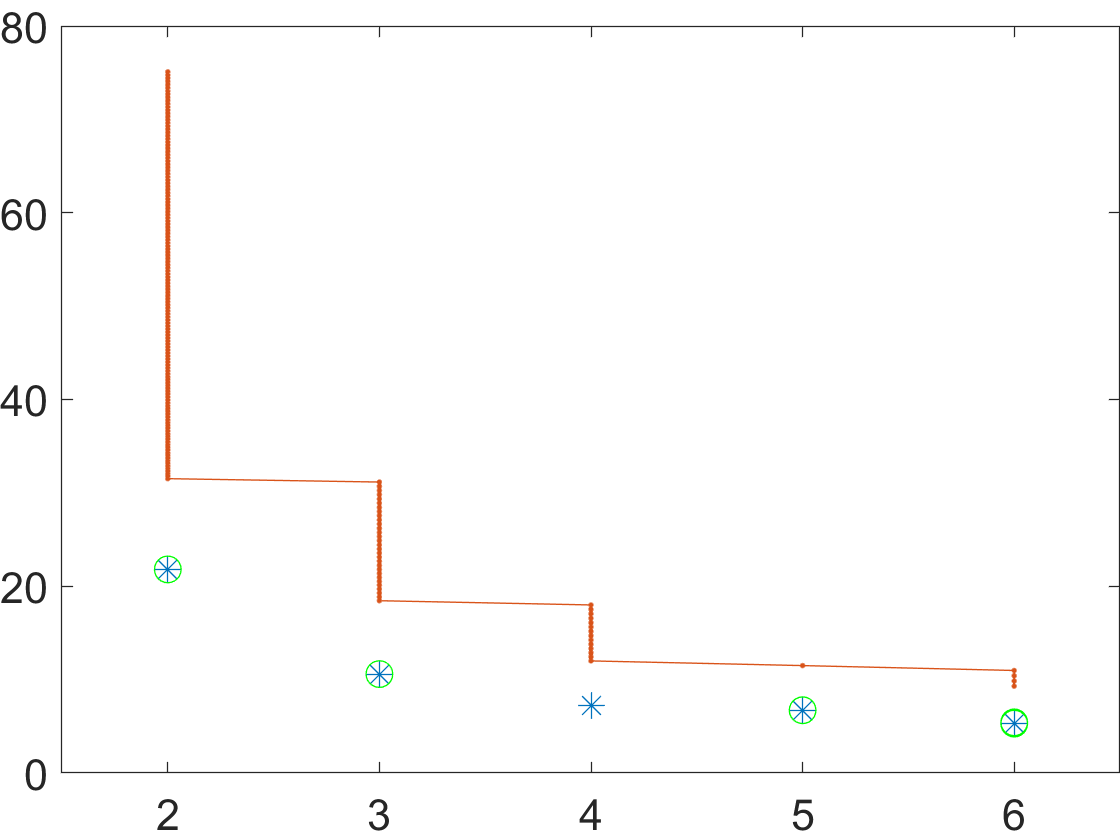} &
				\includegraphics[width=\w]{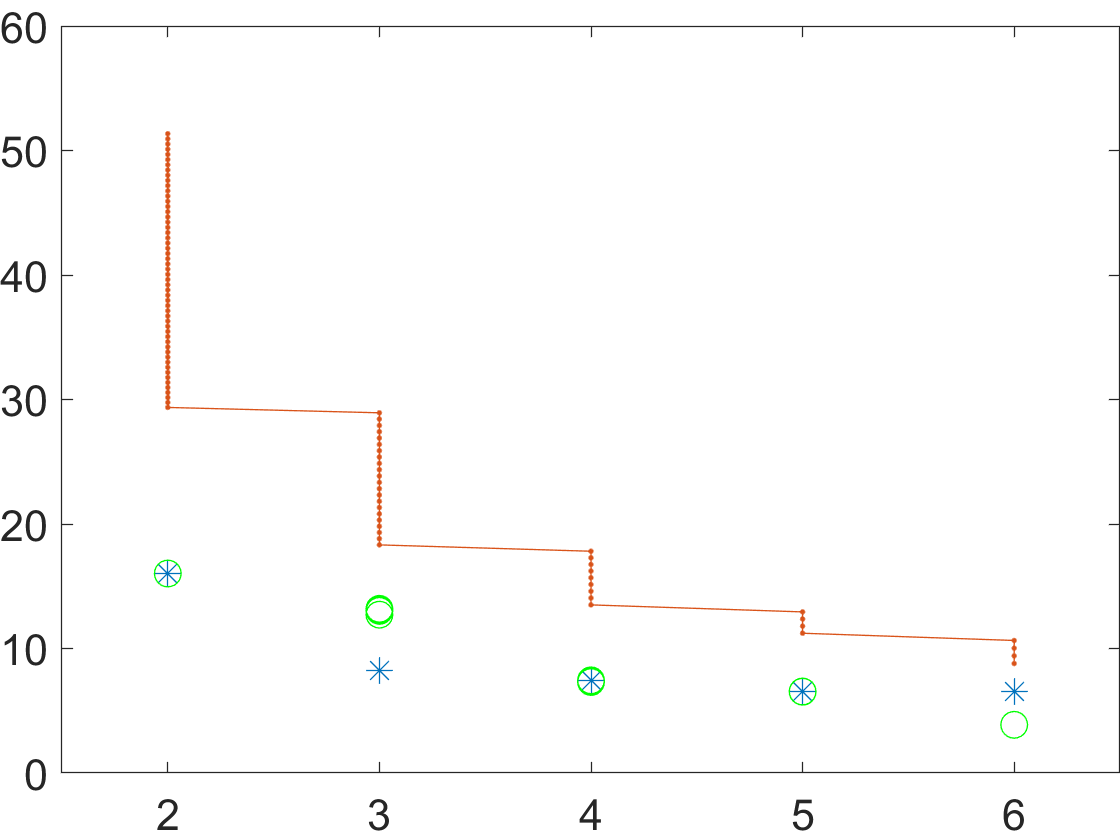}
			\end{tabular}
		\end{center}
		\caption{Results on the four MOCAP sequences used in \cite{dai-etal-ijcv-2014}. The regularizers $\Q_2(\mu \rank)$ (green) and $\Q_2(\iota_{K})$ (blue) consistently outperform the nuclear norm result (red) for all ranks despite being susceptible to local minima.)}
		\label{fig:mocap}
	\end{figure}
}

\section{Acknowledgments}
The first two authors are indebted to support from eSSENCE as well as the Crafoord foundation.
The third author would like to thank himself for a job well done.
His work has been funded by the the Swedish Research Council (grant no.\ 2018-05375) and by the Wallenberg AI, Autonomous Systems and Software Program (WASP) funded by the Knut and Alice Wallenberg Foundation.

\newpage
\bibliographystyle{plain}
\bibliography{newlib}

\end{document}